\documentclass[10pt]{amsart}
\usepackage{amsmath,amssymb,amsthm,array,hyperref,color,cite,graphicx,multirow,pdflscape,mathrsfs}
\usepackage[all,arc]{xypic}
\usepackage[all]{xypic}

\newtheorem{thm}{Theorem}[section]
\newtheorem{cor}[thm]{Corollary}
\newtheorem{lem}[thm]{Lemma}
\newtheorem{prop}[thm]{Proposition}
\theoremstyle{definition}

\newtheorem{df}[thm]{Definition}

\newtheorem{rmk}[thm]{Remark}

\newcommand{\R}{\mathbb R}
\newcommand{\C}{\mathbb C}
\newcommand{\Z}{\mathbb Z}

\newcommand{\Sph}{\mathbb S}
\newcommand{\cat}[1]{\textup{\textbf{{#1}}}}

\newcommand{\Map}{\textup{Map}}

\newcommand{\Tot}{\textup{Tot}}

\newcommand{\id}{\textup{id}}
\newcommand{\colim}{\textup{colim}\,}

\newcommand{\holim}{\textup{holim}\,}

\newcommand{\ra}{\longrightarrow}
\newcommand{\la}{\longleftarrow}
\newcommand{\sma}{\wedge}

\newcommand{\ti}{\widetilde}
\newcommand{\simar}{\overset\sim\longrightarrow}
\newcommand{\congar}{\overset\cong\longrightarrow}
\newcommand{\lcongar}{\overset\cong\longleftarrow}
\newcommand{\mc}{\mathcal}
\newcommand{\op}{\textup{op}}

\newcommand{\sd}{\textup{sd}}

\newcommand{\sk}{\textup{Sk}}

\newcommand{\cyc}{\textup{cyc}}

\renewcommand{\uline}{\underline}
\newcommand{\ev}{\textup{ev}}
\newcommand{\ob}{\textup{ob}}

\newcommand{\sh}{\textup{sh}}

\definecolor{grey}{gray}{0.4}

\title[The topological Hochschild homology of $DX$]{Cyclotomic structure in the topological Hochschild homology of $DX$}
\author{Cary Malkiewich}

\begin{document}

\maketitle
\begin{abstract}
Let $X$ be a finite CW complex, and let $DX$ be its dual in the category of spectra. We demonstrate that the Poincar\'e/Koszul duality between $THH(DX)$ and the free loop space $\Sigma^\infty_+ LX$ is in fact a genuinely $S^1$-equivariant duality that preserves the $C_n$-fixed points. Our proof uses an elementary but surprisingly useful rigidity theorem for the geometric fixed point functor $\Phi^G$ of orthogonal $G$-spectra.
\end{abstract}

\parskip 0ex
\tableofcontents
\parskip 2ex

\section{Introduction.}

Topological Hochschild homology ($THH$) is a powerful and computable invariant of rings and ring spectra. Like ordinary Hochschild homology, it is built by a cyclic bar construction on the ring $R$, but with the tensor products of abelian groups $R \otimes_Z R$ replaced by smash products of spectra $R \sma_{\Sph} R$.

This construction was originally developed by B\" okstedt \cite{bokstedt1985topological}, using ideas of Goodwillie and Waldhausen. The result is a spectrum $THH(R)$ with a circle action. Out of its fixed points one builds topological cyclic homology $TC(R)$, a very close approximation to the algebraic $K$-theory spectrum $K(R)$. This machinery has been tremendously successful at advancing our understanding of $K(R)$ when $R$ is a discrete ring, and Waldhausen's functor $A(X) = K(\Sigma^\infty_+ \Omega X)$ for any space $X$, to say nothing of the $K$-theory of other ring spectra. The $THH$ construction is also of intrinsic interest when one studies topological field theories, and $TC$ appears to be an analog of ``crystalline cohomology'' from algebraic geometry.

In this paper we use $THH$ to study the ring spectrum $DX$, the Spanier-Whitehead dual of a finite CW complex.
We are motivated by classical work on the Hochschild homology of the cochains $C^*(X)$.
In \cite{jones1987cyclic}, Jones proved that when $X$ is simply-connected there is an isomorphism
\[ HH_*(C^*(X)) \cong H^*(LX) \]
where $LX$ is the space of free loops in $X$, and all homology is taken with field coefficients.
We investigate a lift of this theorem to spectra.
Namely, the functional dual of $THH(DX)$ is equivalent to $\Sigma^\infty_+ LX$ when $X$ is finite and simply-connected:
\begin{equation}\label{intro_thh_dual_map}
D(THH(DX)) \simeq \Sigma^\infty_+ LX \simeq THH(\Sigma^\infty_+ \Omega X)
\end{equation}
This was observed by Cohen in the course of some string-topology calculations. Kuhn proved a more general statement for the tensor of the commutative ring $DX$ with any unbased finite complex $K$, not just the circle $S^1$ \cite{kuhn2004mccord}.
The $THH$ duality (\ref{intro_thh_dual_map}) was also extended by Campbell from $(\Sigma^\infty_+ \Omega X,DX)$ to other pairs of Koszul-dual ring spectra \cite{campbell2014derived}. These generalizations can also be seen as special cases of the Poincar\'e/Koszul duality theorem of Ayala and Francis \cite{ayalapoincare}.


If $X = M$ is a closed smooth manifold, we refer to (\ref{intro_thh_dual_map}) as Atiyah duality for the infinite-dimensional manifold $LM$.
Classical Atiyah duality is an equivalence of ring spectra $M^{-TM} \simeq DM$, where $M^{-TM}$ has the intersection product described in \cite{cohen2002homotopy}. If $K$ is a finite set, the $K$-fold multiplicative norm of $M^{-TM}$ is
\[ N^K (M^{-TM}) = \sma^k (M^{-TM}) = (M^k)^{-TM^{\oplus k}} = \Map(K,M)^{-T\Map(K,M)} \]
By analogy, we define the ``Thom spectrum'' of the infinite-dimensional virtual bundle $-TLM$ over $LM$ to be the multiplicative $S^1$-norm of $M^{-TM}$:
\[ LM^{-TLM} = \Map(S^1,M)^{-T\Map(S^1,M)} = N^{S^1} (M^{-TM}). \]
By Angeltveit et al. \cite{angeltveit2014relative}, the $THH$ of a commutative ring spectrum is a model for this multiplicative $S^1$-norm, so the duality (\ref{intro_thh_dual_map}) may be interpreted as
\[ D(LM^{-TLM}) \simeq \Sigma^\infty_+ LM \]

Previous work on the duality (\ref{intro_thh_dual_map}) has left open the question of whether it actually preserves any of the fixed points under the circle action.
We address this with the following theorem.
\begin{thm}\label{intro_dual_equivariant_thm}
When $X$ is finite and simply-connected, the map of (\ref{intro_thh_dual_map}) is an equivalence of cyclotomic spectra. It therefore induces equivalences of fixed point spectra
\[ \begin{array}{c}
\Phi^{C_n} D(THH(DX)) \simeq \Phi^{C_n} \Sigma^\infty_+ LX \\
{[D(THH(DX))]^{C_n}} \simeq [\Sigma^\infty_+ LX]^{C_n}
\end{array} \]
for all finite subgroups $C_n \leq S^1$.
\end{thm}
These notions of fixed points are recalled in section \ref{gspectra_basics}. Cyclotomic spectra are recalled in section \ref{cyclotomic}; the main examples are $THH(R)$ and $\Sigma^\infty_+ LX$, and this is the structure which allows us to compute $TC(R)$ and $TC(X)$.

Implicit in the above theorem is the construction of a cyclotomic structure on the dual $D(THH(DX))$. In fact we show that for any associative ring spectrum $R$, the functional dual $D(THH(R))$ comes with a natural \emph{pre-cyclotomic} structure, and in the case of $R = DX$ with $X$ finite and simply-connected, this becomes a cyclotomic structure.

We believe that this theorem suggests deeper connections between Waldhausen's functor $A(X)$ and the algebraic $K$-theory of $DX$.
We will attempt to explore this idea further in future work.

Our work on $THH(DX)$ builds on very recent results of Angeltveit, Blumberg, Gerhardt, Hill, Lawson, and Mandell \cite{angeltveit2012interpreting}, \cite{angeltveit2014relative}, along with the thesis of Martin Stolz \cite{stolz2011equivariant}.
They establish that the cyclic bar construction, in orthogonal spectra, has the same equivariant behavior as B{\"o}kstedt's original construction of topological Hochschild homology \cite{bokstedt1985topological}.
But in many respects, this cyclic bar construction is much simpler.
This leads to simplifications in the theory of $THH$, as well as new results, including those outlined above.

Our proofs also have consequences for the general theory of cyclotomic spectra and $G$-spectra. Let $G$ be a compact Lie group. We prove a rigidity result for the smash powers and geometric fixed points of orthogonal spectra, which appears to be new and of independent interest. Let $\Phi$ be the functor from $k$-tuples of orthogonal $G$-spectra to orthogonal spectra
\[ \Phi(X_1, \ldots, X_k) = \Phi^G X_1 \sma \ldots \sma \Phi^G X_k \]
where $\Phi^G$ is the monoidal geometric fixed point functor of \cite{mandell2002equivariant}.
\begin{thm}\label{intro_rigidity}
Suppose $\eta: \Phi \rightarrow F$ is a natural transformation, and $\eta$ is an isomorphism on every $k$-tuple of free $G$-spectra. Then there are only two natural transformations from $\Phi$ to $F$: the given transformation $\eta$, and zero.
\end{thm}
We emphasize that this theorem applies to point-set functors of orthogonal spectra, not to functors defined on the homotopy category. It is designed to prove that certain point-set constructions strictly agree, thereby eliminating the need to construct coherence homotopies between them.

The rigidity theorem has a host of technical corollaries. Here are two of them.
\begin{cor}
For $G$ a finite group, the Hill-Hopkins-Ravenel diagonal map
\[ \Phi^H X \overset\Delta\ra \Phi^G N_H^G X \]
is the only nonzero natural transformation from $\Phi^H X$ to $\Phi^G N^G_H X$.
\end{cor}
This also applies to the subcategory of cofibrant spectra, giving an easy proof that the diagonal isomorphism constructed by Stolz \cite{brun2016equivariant} agrees with the one constructed by Hill, Hopkins, and Ravenel \cite{hhr}.
\begin{cor}
For $G$ a compact Lie group, the commutation map
\[ \Phi^G X \sma \Phi^G Y \overset\alpha\ra \Phi^G(X \sma Y) \]
is the only such natural transformation that is nonzero.
\end{cor}
The rigidity theorem gives a useful framework for understanding how multiplicative structure interacts with cyclotomic structure in orthogonal spectra. Motivated by Kaledin's ICM address \cite{kaledin2010motivic}, we use Thm \ref{intro_rigidity} to place certain tensors and internal homs into the model category of cyclotomic spectra \cite{blumberg2013homotopy}. In particular, we get 
\begin{cor}
The homotopy category of cyclotomic spectra is tensor triangulated.
\end{cor}
Barwick and Glasman have recently extended this program further, see \cite{barwick2016cyclonic}.

The paper is organized as follows. In section 2 we review the theory of cyclic spaces and spectra. In section 3 we review orthogonal $G$-spectra, and prove Theorem \ref{intro_rigidity}. In section 4 we combine the previous two sections and develop the norm model of $THH$ following \cite{angeltveit2014relative}. In section 5 we study the interaction of multiplicative structure and cyclotomic structure, proving Theorem \ref{intro_dual_equivariant_thm}.

The author is grateful to acknowledge Andrew Blumberg, Jon Campbell, Ralph Cohen, and Randy McCarthy for several helpful and inspiring conversations throughout this project. He thanks Nick Kuhn for insightful comments on the first version of the paper, and the anonymous referee for a very close reading that substantially improved the exposition throughout. This paper represents a part of the author's Ph.D. thesis, written under the direction of Ralph Cohen at Stanford University.

\section{Review of cyclic spaces.}\label{sec:cyclic}

A \emph{cyclic set} is a simplicial set with extra structure, which allows the geometric realization to carry a natural $S^1$-action \cite{connes1983cohomologie}.
Similarly one may define \emph{cyclic spaces} and \emph{cyclic spectra}.
In this section we collect together the main results of the theory of cyclic spaces, and their extensions to cocyclic spaces.
We also describe (co)cyclic orthogonal spectra, though we defer the study of their equivariant behavior to section \ref{sec:bar}.
This section is all standard material from \cite{dhk}, \cite{jones1987cyclic}, \cite{bhm}, and \cite{madsen_survey} or a straightforward generalization thereof, but we make an effort to be definite and explicit in areas where our later proofs require it.
We will also be brief; the reader seeking more complete proofs is referred to the author's thesis \cite{malkiewich2014duality}.

\subsection{The category $\mathbf\Lambda$ and the natural circle action.}\label{cyclic_basics}

Recall that $\mathbf\Delta$ is a category with one object $[n] = \{0,1,\ldots,n\}$ for each $n \geq 0$.
The morphisms $\mathbf\Delta([m],[n])$ are the functions $f: [m] \rightarrow [n]$ which preserve the total ordering.
It is generated by the \emph{coface maps} and \emph{codegeneracy maps}
\[
\begin{tabular}{ccc}
$d^i: [n-1] \ra [n], \quad 0 \leq i \leq n$ && $s^i: [n+1] \ra [n], \quad 0 \leq i \leq n$  \\
$d^i(j) = \left\{ \begin{array}{ccc} j &\textup{ if }& j < i \\ j+1 &\textup{ if }& j \geq i \end{array} \right.$
&& $s^i(j) = \left\{ \begin{array}{ccc} j &\textup{ if }& j \leq i \\ j-1 &\textup{ if }& j > i \end{array} \right.$
\end{tabular}
\] 
A simplicial object of $\cat C$ is a contravariant functor $X_\bullet: \mathbf\Delta^\op \to \cat C$. We are interested in the case where $\cat C$ is based spaces or orthogonal spectra.
Any simplicial object $X_\bullet$ has a canonical presentation
\[ \bigvee_{m,n} \mathbf\Delta(\bullet,[m])_+ \sma \mathbf\Delta([m],[n])_+ \sma X_n \rightrightarrows \bigvee_n \mathbf\Delta(\bullet,[n])_+ \sma X_n \rightarrow X_\bullet \]
There is a geometric realization functor $|-|$ taking simplicial spaces to spaces. It is the unique colimit-preserving functor that takes $\Delta[n]$ to $\Delta^n$, the convex hull of the standard basis vectors in $\R^{n+1}$. It turns out that for simplicial based spaces $X_\bullet$, the realization $|X_\bullet|$ is given by either of the two coequalizers
\[ \coprod_{m,n} \Delta^m \times \mathbf\Delta([m],[n]) \times X_n \rightrightarrows \coprod_n \Delta^n \times X_n \ra |X_\bullet| \]

\vspace{-1.5em}
\[ \bigvee_{m,n} \Delta^m_+ \sma \mathbf\Delta([m],[n])_+ \sma X_n \rightrightarrows \bigvee_n \Delta^n_+ \sma X_n \ra |X_\bullet|. \]
When $X_\bullet$ is a simplicial orthogonal spectrum we define $|X_\bullet|$ by the latter of these two formulas.

Connes' cyclic category $\mathbf\Lambda$ has the same objects as $\mathbf\Delta$, but more morphisms. Let $[n]$ denote the free category on the arrows
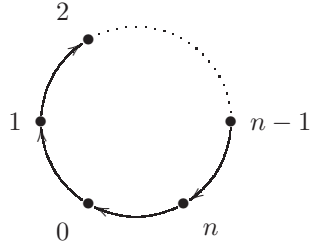
\[ \xy 0;<36pt,0pt>:
a(0)*{\bullet}="n-1"+(0.2,0)*!L{n-1};
a(120)*{\bullet}="2"+(-0.2,0.2)*!DR{2};
a(180)*{\bullet}="1"+(-0.2,0)*!R{1};
a(240)*{\bullet}="0"+(-0.2,-0.2)*!UR{0};
a(300)*{\bullet}="n"+(0.2,-0.2)*!UL{n};
"n-1";"n" **\crv{(0.9786,-0.565)}; ?>*\dir{>};
"n";"0" **\crv{(0,-1.13)}; ?>*\dir{>};
"0";"1" **\crv{(-0.9786,-0.565)}; ?>*\dir{>};
"1";"2" **\crv{(-0.9786,0.565)}; ?>*\dir{>};
0;a(-90) **\dir{}; (0,0)*\ellipse(1)__,=:a(120){.};
\endxy \]
The geometric realization $|N_\bullet [n]|$ of the nerve of the category $[n]$ is homotopy equivalent to the circle.
The set $\mathbf\Lambda([m],[n])$ consists of those functors $[m] \rightarrow [n]$ which give a degree 1 map on the geometric realizations. This is generated by maps in $\mathbf\Delta$ plus a \emph{cycle map} $\tau_n: [n] \rightarrow [n]$ for each $n \geq 0$:
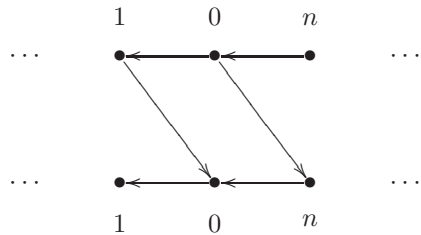
\[ \xy 0;<12pt,0pt>:
(-3,0)*{\bullet}="U1"+(0,1)*!D{1},
(0,0)*{\bullet}="U0"+(0,1)*!D{0},
(3,0)*{\bullet}="Un"+(0,1)*!D{n},
"U1"-(3,0)*{\cdots},
"Un"+(3,0)*{\cdots},
(-3,-4)*{\bullet}="D1"+(0,-1)*!U{1},
"D1"+(3,0)*{\bullet}="D0"+(0,-1)*!U{0},
"D0"+(3,0)*{\bullet}="Dn"+(0,-1)*!U{n},
"D1"-(3,0)*{\cdots},
"Dn"+(3,0)*{\cdots},
"U0";"U1" **\dir{-}; ?>*\dir{>};
"Un";"U0" **\dir{-}; ?>*\dir{>};
"D0";"D1" **\dir{-}; ?>*\dir{>};
"Dn";"D0" **\dir{-}; ?>*\dir{>};
"U1";"D0" **\dir{-}; ?>*\dir{>};
"U0";"Dn" **\dir{-}; ?>*\dir{>};
\endxy \]
We may also generate $\mathbf\Lambda$ by $\mathbf\Delta$ and an \emph{extra degeneracy map} $s^{n+1}: [n+1] \rightarrow [n]$ for each $n \geq 0$, corresponding to the functor $[n+1] \rightarrow [n]$ pictured below:
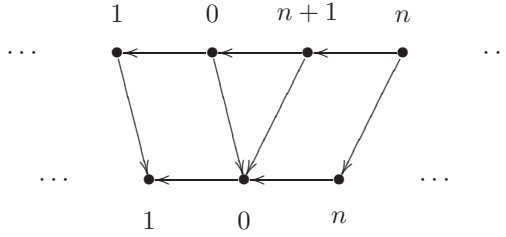
\[ \xy 0;<12pt,0pt>:
(-3,0)*{\bullet}="U1"+(0,1)*!D{1},
(0,0)*{\bullet}="U0"+(0,1)*!D{0},
(3,0)*{\bullet}="Un+1"+(0,1)*!D{n+1},
(6,0)*{\bullet}="Un"+(0,1)*!D{n},
"U1"-(3,0)*{\cdots},
"Un"+(3,0)*{\cdots},
(-2,-4)*{\bullet}="D1"+(0,-1)*!U{1},
"D1"+(3,0)*{\bullet}="D0"+(0,-1)*!U{0},
"D0"+(3,0)*{\bullet}="Dn"+(0,-1)*!U{n},
"D1"-(3,0)*{\cdots},
"Dn"+(3,0)*{\cdots},
"U0";"U1" **\dir{-}; ?>*\dir{>};
"Un+1";"U0" **\dir{-}; ?>*\dir{>};
"Un";"Un+1" **\dir{-}; ?>*\dir{>};
"D0";"D1" **\dir{-}; ?>*\dir{>};
"Dn";"D0" **\dir{-}; ?>*\dir{>};
"U1";"D1" **\dir{-}; ?>*\dir{>};
"U0";"D0" **\dir{-}; ?>*\dir{>};
"Un+1";"D0" **\dir{-}; ?>*\dir{>};
"Un";"Dn" **\dir{-}; ?>*\dir{>};
\endxy \]
We note that a morphism $f \in \mathbf\Lambda([m],[n])$ is determined by the underlying map of sets $\Z/(m+1) \to \Z/(n+1)$, unless this map of sets is constant, in which case $f$ is determined by which arrow in $[m]$ is sent to a nontrivial arrow in $[n]$.

\begin{df}
A \emph{cyclic based space} is a functor $X_\bullet: \mathbf\Lambda^\op \to \cat{Top}_*$.
The \emph{geometric realization} $|X_\bullet|$ is defined by restricting $X_\bullet$ to $\mathbf\Delta^\op$ and taking the geometric realization of the resulting simplicial space.
\end{df}

\begin{thm}[e.g. \cite{dhk}]
The geometric realization $|X_\bullet|$ of a cyclic based space $X$ carries a natural based $S^1$-action.
\end{thm}

\begin{proof}
The cyclic space $X_\bullet$ is a colimit of representable cyclic sets
\[ \Lambda[n] = \mathbf\Lambda(-,[n]) \]
So, it suffices to prove that the space
\[ \Lambda^n := |\Lambda[n]| \]
has an $S^1$ action for all $n$, commuting with the action of the category $\mathbf\Lambda$. By a combinatorial argument, we have homeomorphisms $\Lambda^n \cong S^1 \times \Delta^n$, and we define an $S^1$ action by translation on the first coordinate. These actions commute with the action of $\mathbf\Lambda$, and so they pass to the realization.
We draw a few special cases of $\Lambda^n$ and how it compares to the simplicial circle times $\Delta^n$.

\resizebox{\textwidth}{!}{$
\begin{array}{c|cc}
n & \Lambda^n & \Delta[1]/\partial \times \Delta[n] \\\hline
0 &
	{\xy 0;<15pt,0pt>:
		(0,0)*{\bullet}="00"+(-1,0)*!R{(0,0)},
		(4,0)*{\bullet}="10"+(1,0)*!L{(1,0)},
		"00","10" **\dir{-}; ?(.55)*\dir{>};
	\endxy}
&
	{\xy 0;<15pt,0pt>:
		(0,0)*{\bullet}="00"+(-1,0)*!R{(0,0)},
		(4,0)*{\bullet}="10"+(1,0)*!L{(1,0)},
		"00","10" **\dir{-}; ?(.55)*\dir{>};
	\endxy}
	\\[.5em]\hline
1 &
	{\xy 0;<15pt,0pt>:
		(0,0)*{\bullet}="00"+(-1,0)*!R{(0,0)},
		(4,0)*{\bullet}="10"+(1,0)*!L{(1,0)},
		(0,-3)*{\bullet}="01"+(-1,0)*!R{(0,1)},
		(4,-3)*{\bullet}="11"+(1,0)*!L{(1,1)},
		"00";"01" **\dir{-}; ?(.55)*\dir{>};
		"10";"11" **\dir{-}; ?(.55)*\dir{>};
		"00";"10" **\dir{-}; ?(.55)*\dir{>};
		"01";"11" **\dir{-}; ?(.55)*\dir{>};
		"01";"10" **\dir{-}; ?(.55)*\dir{>};
	\endxy}
&
	{\xy 0;<15pt,0pt>:
		(0,0)*{\bullet}="00"+(-1,0)*!R{(0,0)},
		(4,0)*{\bullet}="10"+(1,0)*!L{(1,0)},
		(0,-3)*{\bullet}="01"+(-1,0)*!R{(0,1)},
		(4,-3)*{\bullet}="11"+(1,0)*!L{(1,1)},
		"00";"01" **\dir{-}; ?(.55)*\dir{>};
		"10";"11" **\dir{-}; ?(.55)*\dir{>};
		"00";"10" **\dir{-}; ?(.55)*\dir{>};
		"01";"11" **\dir{-}; ?(.55)*\dir{>};
		"00";"11" **\dir{-}; ?(.55)*\dir{>};
	\endxy}
	\\[5em]\hline\\[-1em]
2 &
	{\xy 0;<15pt,0pt>:
		(0,0)*{\bullet}="00"+(-1,0)*!R{(0,0)},
		(7,0)*{\bullet}="10"+(1.7,0)*!L{(1,0)},
		(1,2)*{\bullet}="01"+(-1,0)*!R{(0,1)},
		(8,2)*{\bullet}="11"+(1,0)*!L{(1,1)},
		(2,-2)*{\bullet}="02"+(-1,0)*!R{(0,2)},
		(9,-2)*{\bullet}="12"+(1,0)*!L{(1,2)},
		"00";"01" **\dir{-}; ?(.55)*\dir{>};
		"01";"02" **\dir{-}; ?(.45)*\dir{>};
		"00";"02" **\dir{-}; ?(.55)*\dir{>};
		"10";"11" **\dir{-}; ?(.55)*\dir{>};
		"11";"12" **\dir{-}; ?(.55)*\dir{>};
		"10";"12" **\dir{-}; ?(.55)*\dir{>};
		"00";"10" **\dir{--}; ?(.55)*\dir{>};
		"01";"11" **\dir{-}; ?(.55)*\dir{>};
		"02";"12" **\dir{-}; ?(.55)*\dir{>};
		"01";"10" **\dir{--}; ?(.55)*\dir{>};
		"02";"11" **\dir{-}; ?(.60)*\dir{>};
		"02";"10" **\dir{--}; ?(.55)*\dir{>};
	\endxy}
&
	{\xy 0;<15pt,0pt>:
		(0,0)*{\bullet}="00"+(-1,0)*!R{(0,0)},
		(7,0)*{\bullet}="10"+(1.7,0)*!L{(1,0)},
		(1,2)*{\bullet}="01"+(-1,0)*!R{(0,1)},
		(8,2)*{\bullet}="11"+(1,0)*!L{(1,1)},
		(2,-2)*{\bullet}="02"+(-1,0)*!R{(0,2)},
		(9,-2)*{\bullet}="12"+(1,0)*!L{(1,2)},
		"00";"01" **\dir{-}; ?(.55)*\dir{>};
		"01";"02" **\dir{-}; ?(.55)*\dir{>};
		"00";"02" **\dir{-}; ?(.55)*\dir{>};
		"10";"11" **\dir{-}; ?(.55)*\dir{>};
		"11";"12" **\dir{-}; ?(.55)*\dir{>};
		"10";"12" **\dir{-}; ?(.55)*\dir{>};
		"00";"10" **\dir{--}; ?(.55)*\dir{>};
		"01";"11" **\dir{-}; ?(.55)*\dir{>};
		"02";"12" **\dir{-}; ?(.55)*\dir{>};
		"00";"11" **\dir{--}; ?(.55)*\dir{>};
		"01";"12" **\dir{-}; ?(.55)*\dir{>};
		"00";"12" **\dir{--}; ?(.55)*\dir{>};
	\endxy}
	\\[3.5em]\hline
\end{array}
$}
\end{proof}

\subsection{Skeleta and latching objects.}\label{sec:skeleta_latching}

When $X_\bullet$ is a simplicial space, the $n$th \emph{skeleton} $\sk_n X_\bullet$ is obtained by restricting $X_\bullet$ to the subcategory of $\mathbf\Delta^\op$ on the objects $0,\ldots,n$ and then taking a left Kan extension back.
The geometric realization of each skeleton is obtained from the previous one by a pushout square
\begin{equation}\label{latching_square}
\xymatrix{
L_n X \times \Delta^n \cup_{L_n X \times \partial\Delta^n} X_n \times \partial\Delta^n \ar[r] \ar[d] & X_n \times \Delta^n \ar[d] \\
|\sk_{n-1} X_\bullet| \ar[r] & |\sk_n X_\bullet| }
\end{equation}
Here $L_n X$ is the $n$th \emph{latching object}, the subspace of $X_n$ consisting of all points in the images of some degeneracy map $s_i: X_{n-1} \to X_n$, $0 \leq i \leq n-1$.
Alternatively, to each proper subset $S \subseteq \{0,1,\ldots,n\}$ that contains 0, we define a map of totally ordered sets $[n] \to S$ by rounding down to the nearest element of $S$.
This makes $X_S$ into a subspace of $X_n$, and the colimit of these subspaces under inclusions $S \subset T$ gives the subspace $L_n X$.

\begin{df}
$X_\bullet$ is \emph{Reedy $q$-cofibrant} if each $L_n X \to X_n$ is a cofibration in the Quillen model structure on based spaces.
$X_\bullet$ is \emph{Reedy $h$-cofibrant} if each $L_n X \to X_n$ is a classical cofibration, i.e. a map satisfying the unbased homotopy extension property.
\end{df}
We have stated these definitions for based spaces, but they also apply to orthogonal spectra.
There is a standard cofibrantly generated model structure that provides the $q$-cofibrations, while the $h$-cofibrations are defined as maps having the homotopy extension property with respect to the cylinders $X \sma I_+$ \cite{mmss}.
So the following standard theorem applies to both spaces and spectra, with either notion of ``cofibration:''
\begin{prop}\label{reedy_cofibrant_spectra}
If $X_\bullet$ is Reedy cofibrant then $|X_\bullet|$ is cofibrant.
If both $X_\bullet$ and $Y_\bullet$ are Reedy cofibrant, then any map $X_\bullet \simar Y_\bullet$ that is an equivalence on each simplicial level induces an equivalence $|X_\bullet| \simar |Y_\bullet|$.
\end{prop}
\begin{proof}
For simplicial spaces, the proof is an induction up the cube-shaped diagram defining $L_n X$, using the usual pushout and pushout-product properties for cofibrations.
The use of unbased $h$-cofibrations was critical $-$ the theorem is not true with based $h$-cofibrations, unless all the spaces are well-based.

For orthogonal spectra and $q$-cofibrations the proof is largely the same.
For $h$-cofibrations of orthogonal spectra, the theorem is a little surprising since we do not assume any of the spectra involved are well-based.
The hardest piece of the proof is the statement that if $f: K \to L$ is a relative CW complex and $g: A \to X$ is an $h$-cofibration of orthogonal spectra, the pushout-product $f \square g$ is an $h$-cofibration.
This follows from the formal pairing result of Schwa\"nzl and Vogt (\cite{schwanzl2002strong}, Cor 2.9).
\end{proof}

When $X_\bullet$ is a cyclic space, the simplicial skeleton $|\sk_n X_\bullet|$ is of limited utility because it is not closed under the circle action.
So we draw motivation from \cite{berger2011extension} and make the following definitions.
Since it is important, we remark that here and elsewhere we work in the category of compactly generated, weak Hausdorff spaces.
\begin{df}
For $n \geq 0$ we define the $n$th \emph{cyclic skeleton} $\sk_n^\cyc X$ by restricting $X_\bullet$ to the subcategory of $\mathbf\Lambda^\op$ on the objects $0,\ldots,n$ and then taking a left Kan extension back. This may be re-expressed as the coequalizer
\[ \bigvee_{k,\ell \leq n} \mathbf\Lambda(\bullet,[k])_+ \sma \mathbf\Lambda([k],[\ell])_+ \sma X_\ell \rightrightarrows \bigvee_{k \leq n} \mathbf\Lambda(\bullet,[k])_+ \sma X_k \rightarrow \sk^\cyc_n X_\bullet \]
We take the $(-1)$st cyclic skeleton to be the space $X_{-1}$, defined as the equalizer of the degeneracy and extra degeneracy maps:
\[ \sk_{-1}^\cyc X = X_{-1} \rightarrow X_0 \rightrightarrows X_1 \]
\end{df}

\begin{df}
The $n$th \emph{cyclic latching object} $L_n^\cyc X \subset X_n$ is the closed subspace consisting of all points lying in the image of some degeneracy map
\[ s_i: X_{n-1} \ra X_n, \qquad 0 \leq i \leq n \]
The $0$th latching object is also taken to be $\sk_{-1}^\cyc X \subset X_0$ rather than being empty.
\end{df}

The only difference between $L_n X$ and $L_n^\cyc X$ is that the \emph{extra} degeneracy is included in $L_n^\cyc X$.
Equivalently, $L_n^\cyc X$ is the closure of $L_n X$ under the action of the cycle map $t_n$.
It follows that $|\sk_n^\cyc X_\bullet|$ is the closure of $|\sk_n X_\bullet|$ under the circle action.

We briefly prove an equivalent characterization of $L_n^\cyc X$.
Let $[n]$ denote the cycle category with $n+1$ objects from the definition of $\mathbf\Lambda$.
Each inclusion of a nonempty subset $S \subset \{0,\ldots,n\}$ gives a degree 1 functor $[n] \to [|S|-1]$ which rounds down to the nearest element of $S$.
By the cyclic structure of $X$, this gives a map $X_S := X_{|S|-1} \to X_n$.
If $S$ is empty then we define $X_S = X_{-1}$, and define $X_S \to X_n$ by including into $X_0$ and applying any composition of degeneracy maps $X_0 \to X_n$.
\begin{prop}\label{cyclic_latching_cube}
This forms a cube-shaped diagram of subspaces of $X_n$, indexed by the subsets of $\{0,\ldots,n\}$ and inclusions. Restricting to the proper subsets, the colimit of this diagram is $L_n^{\cyc} X$.
\end{prop}
\begin{proof}
If $n < 1$ then this is easy, so we assume $n \geq 1$. It is straightforward to check that our rule respects inclusions of subsets. Each edge of the cube is a standard degeneracy map, which is split by some face map. Since we are working in weak Hausdorff spaces, this implies that each $X_S$ is a closed subspace of $X_n$. To prove that their colimit is equal to their union, it suffices to check $X_S \cap X_T = X_{S \cap T}$. This reduces to the following claim: for each $0 \leq i \leq n$, let $D_i: X_n \to X_n$ be the map induced by the functor $[n] \to [n]$ that sends $i$ to $i-1$ and fixes all other points. Then $X_S$ is precisely the subspace that is fixed by $D_i$ for every $i$ in the complement of $S$.

To prove this when $S$ is nonempty, observe there is a natural projection map $d_S: X_n \to X_S$ induced by the inclusion of $S$ into $[n]$.
Thinking of this as a map $X_n \to X_n$, $X_S$ is precisely the subspace fixed by this projection.
On the other hand, we may write the complement of $S$ as some cyclically ordered set $\{m_1,\ldots,m_k\}$, arranged so that $m_k + 1 \in S$, and then we have the identity
\[ d_S = d_{\{m_k\}^c} \ldots d_{\{m_1\}^c} = D_{m_k} \ldots D_{m_1} \]
Therefore being in $X_S$ is equivalent to being fixed by $D_i$ for all $i \in S^c$.

If $S$ is empty, then $X_\emptyset = X_{-1}$ is contained in every $X_0$ and so is fixed by all the projections $D_i$.
Conversely, anything fixed by all the projections is in every subspace of the form $X_{\{s\}} \cong X_0$.
In particular it lies in $X_{\{0\}}$ and $X_{\{1\}}$.
This gives two points $x_0,x_1 \in X_0$ whose images under the two degeneracy maps are the same point $x \in X_1$.
But each face map splits both degeneracy maps, so $x_0 = x_1$ and this point of $X_0$ lies in the subspace $X_{-1}$.
\end{proof}

Now we give the analogue of the standard pushout square (\ref{latching_square}). We expect this is known, but have not found a reference.

\begin{prop}\label{cyclic_latching_square}
For each $n \geq 0$, there is a natural pushout square of $S^1$-spaces
\begin{equation}
\xymatrix{
L_n^\cyc X \times_{C_{n+1}} \Lambda^n \cup_{L_n^\cyc X \times \partial\Lambda^n} X_n \times_{C_{n+1}} \partial\Lambda^n \ar[r] \ar[d] & X_n \times_{C_{n+1}} \Lambda^n \ar[d] \\
|\sk_{n-1}^\cyc X_\bullet| \ar[r] & |\sk_n^\cyc X_\bullet| }
\end{equation}
for each unbased cyclic space $X_\bullet$, and the obvious variant with smash products when $X_\bullet$ is a based cyclic space.
\end{prop}
\begin{proof}
The square is clearly defined and natural, and the top horizontal map is the inclusion of a subspace.
We treat the case $n = 0$ separately, where the square becomes
\[ \xymatrix{
(L_0^\cyc X \times S^1) \amalg \emptyset \ar[r] \ar[d] & X_0 \times S^1 \ar[d] \\
L_0^\cyc X \ar[r] & |\sk_0^\cyc X_\bullet| } \]
which is easily checked to be a pushout.
For $n \geq 1$, it suffices to check that it is a pushout when $X_\bullet = \Lambda(\bullet,[m])$ is the standard cyclic $m$-simplex.
The square may be rewritten
\[ \xymatrix{
(L_n^\cyc \Lambda[m] \times_{C_{n+1}} \Lambda^n) \amalg (\Lambda_n[m] - L_n^\cyc \Lambda[m]) \times_{C_{n+1}} \partial\Lambda^n \ar[r] \ar[d] & \Lambda_n[m] \times_{C_{n+1}} \Lambda^n \ar[d] \\
|\sk_{n-1}^\cyc \Lambda[m]| \ar[r] & |\sk_n^\cyc \Lambda[m]| } \]
The top map is a disjoint union of some isomorphisms and some nontrivial inclusions. We strike out the isomorphisms without changing whether the square is a pushout:
\[ \xymatrix{
(\Lambda_n[m] - L_n^\cyc \Lambda[m]) \times_{C_{n+1}} \partial\Lambda^n \ar[r] \ar[d] & (\Lambda_n[m] - L_n^\cyc \Lambda[m]) \times_{C_{n+1}} \Lambda^n \ar[d] \\
|\sk_{n-1}^\cyc \Lambda[m]| \ar[r] & |\sk_n^\cyc \Lambda[m]| } \]
The complement of the latching object $L_n^\cyc \Lambda[m]$ consists of maps in $\mathbf\Lambda([n],[m])$ for which the $n+1$ points $0,\ldots,n$ go to distinct points in $0,\ldots,m$.
The $C_{n+1}$-action on these maps is free and each orbit has a unique representative that comes from $\mathbf\Delta([n],[m])$, so we can again simplify the square to
\[ \xymatrix{
(\Delta_n[m] - L_n \Delta[m]) \times \partial\Lambda^n \ar[r] \ar[d] & (\Delta_n[m] - L_n \Delta[m]) \times \Lambda^n \ar[d] \\
|\sk_{n-1}^\cyc \Lambda[m]| \ar[r] & |\sk_n^\cyc \Lambda[m]| } \]
Now one may identify this square as the standard simplicial pushout square for $\Delta[m]$, multiplied by the identity map on $S^1$.
Alternatively, one can enumerate the cells of $|\sk_n^\cyc \Lambda[m]|$ missing from $|\sk_{n-1}^\cyc \Lambda[m]|$ and check that the above map precisely attaches those cells.
So the square is a pushout and the proof is complete.
\end{proof}

As a result, Proposition \ref{reedy_cofibrant_spectra} applies to cyclic spectra whose cyclic latching maps are cofibrations, including the ``$(-1)$st latching map'' $* \to X_{-1}$. One can even check that being Reedy cofibrant in the cyclic sense is stronger than being Reedy cofibrant in the ordinary sense.

We will need to know when $|X_\bullet|$ is a cofibrant as a space with an $S^1$ action:
\begin{df}
If $G$ is a topological group, a map $X \rightarrow Y$ of based $G$-spaces is a \emph{cofibration} if it is a retract of a relative cell complex built out of cells of the form
\[ (G/H \times \partial D^n)_+ \hookrightarrow (G/H \times D^n)_+ \]
with $n \geq 0$ and $H \leq G$ any closed subgroup.
\end{df}
\begin{prop}\label{reedy_cof_implies_realization_is_cof_s1_spaces}
If $X_\bullet$ is a cyclic space, $X_{-1}$ is a cofibrant space, and each cyclic latching map $L_n^\cyc X \to X_n$ is a cofibration of $C_{n+1}$-spaces, then $|X_\bullet|$ is a cofibrant $S^1$-space.
\end{prop}
\begin{proof}
It suffices to show that each map of cyclic skeleta
\[ |\sk_{n-1}^\cyc X| \ra |\sk_n^\cyc X| \]
is an $S^1$-cofibration.
The $(-1)$-skeleton is already assumed to be cofibrant, and it has trivial $S^1$-action, so it is also $S^1$-cofibrant.
For the induction we use the square from Prop. \ref{cyclic_latching_square}:
\begin{equation*}
\xymatrix{
L_n^\cyc X \times_{C_{n+1}} \Lambda^n \cup_{L_n^\cyc X \times \partial\Lambda^n} X_n \times_{C_{n+1}} \partial\Lambda^n \ar[r] \ar[d] & X_n \times_{C_{n+1}} \Lambda^n \ar[d] \\
|\sk_{n-1}^\cyc X_\bullet| \ar[r] & |\sk_n^\cyc X_\bullet| }
\end{equation*}
It suffices to prove that the top horizontal is an $S^1$-cofibration.
Since $L_n^\cyc X \rightarrow X$ is a $C_{n+1}$-cofibration, and $\partial \Lambda^n \rightarrow \Lambda^n$ is a free $S^1$-cofibration, this reduces to proving that the $C_{n+1}$ orbits of the simpler pushout-product
\[ [ (C_{n+1}/C_r \times \partial D^k \rightarrow C_{n+1}/C_r \times D^k)_+ \square (S^1 \times \partial D^\ell \rightarrow S^1 \times D^\ell)_+ ]_{C_{n+1}} \]
is an $S^1$-cofibration.
By associativity of the pushout-product we rewrite this as
\[ [ (C_{n+1}/C_r \times S^1)_+ \sma (\partial D^{k + \ell} \rightarrow D^{k + \ell})_+ ]_{C_{n+1}} \]
which simplifies to
\[ (S^1/C_r)_+ \sma (\partial D^{k + \ell} \rightarrow D^{k + \ell})_+ \]
and this is one of the generating $S^1$-cofibrations.
\end{proof}

\subsection{Fixed points and subdivision.}

We turn our attention to the fixed points $|X_\bullet|^{C_r}$, where $C_r \leq S^1$ is the cyclic subgroup of order $r$.
The $C_r$-fixed points have an action of $S^1/C_r$, which we usually regard as an $S^1$-action by pulling back along the group isomorphism
\[ \rho_r: S^1 \congar S^1/C_r \]
We will recall the standard result that the $C_r$-fixed points of $|X_\bullet|$ are built from the spaces $X_{rk-1}^{C_r}$ for $k \geq 1$. 
One applies a subdivision functor to $X_\bullet$ to obtain a new simplicial space $\sd_r X_\bullet$, whose realization is homeomorphic to $|X_\bullet|$, but with simplicial $C_r$ action, giving a homeomorphism
\[ |X_\bullet|^{C_r} \cong |(\sd_r X_\bullet)^{C_r}| \]
In fact, one may even put $S^1$ actions on everything in sight, and the relevant maps are all equivariant.
We recall the precise definitions and theorems below.

\begin{df}\cite{bhm}
The \emph{$r$-fold edgewise subdivision functor} a map of categories $\mathbf\Delta \overset{\sd_r}\to \mathbf\Delta$ which takes $[k-1]$ to $[rk-1]$.
Each order-preserving map $[m-1] \to [n-1]$ is repeated $r$ times to give a map $[rm-1] \to [rn-1]$.
Given a simplicial space $X$, we let the \emph{$r$-fold edgewise subdivision} $\sd_r X$ denote the simplicial space obtained by composing with $\sd_r$.
\end{df}

\begin{df}
The \emph{$r$-cyclic category} $\mathbf\Lambda_r$ is the subcategory of $\mathbf\Lambda$ on the objects of the form $[rk-1]$, $k \geq 1$, generated by all maps in the image of $\sd_r: \mathbf\Delta \to \mathbf\Delta$ in addition to the cycle maps.
When working in $\mathbf\Lambda_r$ we relabel the object $[rk-1]$ as $[k-1]$. Equivalently, \mbox{$\mathbf\Lambda_r([k-1],[n-1])$} consists of all non-decreasing functions $f: \Z \to \Z$ such that $f(x+k) = f(x) + n$, up to the equivalence relation $f \sim f + rn$.
\end{df}

\begin{prop}
If $X_\bullet$ is a cyclic space, its $r$-fold subdivision $\sd_r X_\bullet$ is naturally an $r$-cyclic object in $C_r$-spaces. The $C_r$-action is generated by $t_{rn-1}^n$ at simplicial level $n-1$.
\end{prop}

\begin{prop}[\cite{bhm}, 1.1]\label{prop:cyclic_diagonal}
There is a natural diagonal homeomorphism
\[ |\sd_r X_\bullet| \overset{D_r}\ra |X_\bullet| \]
which sends each $(k-1)$-simplex in $X_{rk-1}$ to the corresponding $(rk-1)$-simplex in $X_{rk-1}$ by the diagonal
\[ (u_0,\ldots,u_{k-1}) \mapsto \left(\frac1r u_0,\ldots,\frac1r u_{k-1},\frac1r u_0,\ldots,\frac1r u_{k-1},\ldots \right) \]
\end{prop}

\begin{thm}[\cite{bhm}, 1.6-1.8,1.11]
The realization of any $r$-cyclic space carries a natural $S^1$-action.
The generator of the subgroup $C_r \leq S^1$ acts by the simplicial map $t_{rn-1}^n$.
If $X_\bullet$ is a cyclic space, the diagonal homeomorphism $D_r$ is $S^1$-equivariant.
\end{thm}

Now that we can freely replace $|X_\bullet|$ with $|\sd_r X_\bullet|$ as an $S^1$ space, we see that the $C_r$-fixed points can be built from  the levelwise fixed points $(\sd_r X_\bullet)^{C_r}$. These levelwise fixed points are \emph{a priori} an $r$-cyclic space, but they are actually a cyclic space because they factor through the following quotient functor.

\begin{df}
The quotient functor
\[ P_r: \mathbf\Lambda_r([m-1],[n-1]) \ra \mathbf\Lambda([m-1],[n-1]) \]
takes a function $f: \Z \to \Z$ up to $f \sim f + rn$ and mods out by the stronger equivalence relation $f \sim f + n$.
\end{df}

We always consider $(\sd_r X_\bullet)^{C_r}$ to be a cyclic space, reserving the notation $P_r (\sd_r X_\bullet)^{C_r}$ for the corresponding $r$-cyclic space. With these conventions, the isomorphism between $|X_\bullet|^{C_r}$ and $|(\sd_r X_\bullet)^{C_r}|$ is $S^1$-equivariant:

\begin{prop}[\cite{bhm}, 1.10-1.12]\label{cyclic_fixed_points}\label{cyclic_fixed_points_compatibility}
The passage between cyclic and $r$-cyclic structures on $\sd_r X_\bullet$ and $(\sd_r X_\bullet)^{C_r}$, together with the diagonal of Proposition \ref{prop:cyclic_diagonal}, give natural $S^1$-equivariant homeomorphisms
\[ |(\sd_r X_\bullet)^{C_r}| \cong \rho_r^*|P_r (\sd_r X_\bullet^{C_r})| \cong \rho_r^*(|\sd_r X_\bullet|^{C_r}) \overset{D_r}\ra \rho_r^*(|X_\bullet|^{C_r}) \]
making the following triangle commute:
\[ \xymatrix{
|(\sd_{rs} X_\bullet)^{C_{rs}}| \ar[d]^\cong \ar[rd]^\cong & \\
\rho_r^* |(\sd_{s} X_\bullet)^{C_s}|^{C_r} \ar[r]^\cong & \rho_{rs}^* |X_\bullet|^{C_{rs}} } \]
\end{prop}

\subsection{Cocyclic spaces.}

The previous section dualizes easily.
Recall that a cosimplicial object is a covariant functor $X^\bullet: \mathbf\Delta \to \cat C$.
This can be canonically expressed as an equalizer
\[ X^\bullet \rightarrow \prod_n \Map(\Delta_\bullet[n],X^n)
 \rightrightarrows \prod_{m,n} \Map(\Delta_\bullet[m] \times \mathbf\Delta(m,n),X^n) \]
and so a right adjoint out of cocyclic spaces is determined by what it does to the cosimplicial space $\Map(\Delta_\bullet[n],X^n)$.
The totalization is the unique limit-preserving functor to spaces which takes $\Map(\Delta_\bullet[n],A)$ to $\Map(\Delta^n,A)$.
It is given by the equalizer
\[ \Tot(X^\bullet) \ra \prod_n \Map(\Delta^n,X^n) \rightrightarrows \prod_{m,n} \Map(\Delta^m \times \mathbf\Delta(m,n),X^n) \]
The totalization of a cosimplicial orthogonal spectrum is given by the same formula.

If $X^\bullet$ is not just cosimplicial, but cocyclic, then its totalization is the equalizer
\[ \Tot(X^\bullet) \rightarrow \prod_n \Map(\Lambda^n,X^n)
 \rightrightarrows \prod_{m,n} \Map(\Lambda^m \times \mathbf\Lambda(m,n),X^n) \]
which is enough to prove
\begin{prop}
The totalization of a cocyclic space $X^\bullet$ carries a natural $S^1$-action.
Similarly, the totalization of an $r$-cocyclic space $Y^\bullet$ carries a natural $S^1$-action, in which the action of $C_r \leq S^1$ is the totalization of a cosimplicial map.
\end{prop}

In the special case of $X^\bullet = \Map(E_\bullet,X)$, where $E_\bullet$ is a cyclic space, the canonical homeomorphism
\[ \Tot(X^\bullet) \cong \Map(|E_\bullet|,X) \]
is $S^1$-equivariant.
A useful example to keep in mind is $\Map(S^1_\bullet,X)$, the standard cosimplicial model for the free loop space $LX$.

Next we recall Reedy fibrancy, which we will only need for cosimplicial spectra (as opposed to spaces). We recall that the construction of the latching map $L_n X \to X_n$ for simplicial spectra dualizes to that of the matching map $X^n \to M_n X$ for cosimplicial spectra. We say that $X^\bullet$ is \emph{Reedy fibrant} if these matching maps are fibrations in the stable model structure on orthogonal spectra. The standard analogue of Proposition \ref{reedy_cofibrant_spectra} is:
\begin{prop}
A weak equivalence of Reedy fibrant cosimplicial spectra induces a weak equivalence on the totalizations.
\end{prop}
As expected, one can always replace a cosimplicial spectrum by a Reedy fibrant one that is equivalent on every cosimplicial level. In this paper, we will only use Reedy fibrant cosimplicial spectra of the form $F(X_\bullet,Y)$, where $F(-,-)$ is the internal hom in orthogonal spectra, $Y$ is a fibrant spectrum, and $X_\bullet$ is a Reedy $q$-cofibrant simplicial spectrum. It is straightforward to verify from the properties of the model structure in \cite{mmss} that such an $F(X_\bullet,Y)$ is always Reedy fibrant.

Finally, any cocyclic space $X^\bullet$ may be composed with $\sd_r$ to give an $r$-cocyclic space $\sd_r X^\bullet$. As before, the fixed points of $\Tot(X^\bullet)$ can be recovered as $\Tot((\sd_r X^\bullet)^{C_r})$:
\begin{prop}
If $X^\bullet$ is a cosimplicial space, there is a natural diagonal homeomorphism
\[ \Tot(X^\bullet) \overset{D_r}\ra \Tot(\sd_r X^\bullet) \]
If $X^\bullet$ is cocyclic, $D_r$ is $S^1$-equivariant.
\end{prop}
\begin{prop}\label{cocyclic_fixed_points}
If $X^\bullet$ is a cocyclic space, then $(\sd_r X^\bullet)^{C_r}$ may be regarded as a cocyclic space, and there are natural $S^1$-equivariant homeomorphisms
\[ \Tot((\sd_r X^\bullet)^{C_r}) \cong \rho_r^*\Tot(\sd_r X^\bullet)^{C_r} \overset{D_r^{C_r}}\la \rho_r^*\Tot(X^\bullet)^{C_r} \]
\end{prop}
The proofs are easy dualizations or direct copies of the proofs for cyclic spaces.

\subsection{The suspension spectrum of $LX$.}

We end this section with a more concrete example.
If $X$ is any unbased space, then $\Map(S^1_\bullet,X)$ is a cocyclic space.
We add a disjoint basepoint, and smash every level with the sphere spectrum, yielding a cocyclic spectrum
\[ \Sph \sma \Map(S^1_\bullet,X)_+ = \Sigma^\infty_+ X^{\bullet + 1} \]
It is not hard to check that there is a natural map
\begin{equation}\label{eq:lx_tot_interchange}
\Sigma^\infty_+ LX = \Sigma^\infty_+ \Tot(X^{\bullet + 1}) \ra \Tot(\Sigma^\infty_+ X^{\bullet + 1})
\end{equation}
given by the interchange
\begin{equation}\label{eq:lx_interchange}
\Sph \sma \prod_k \Map_*(\Delta^k_+,X^k_+)
 \ra \prod_k \Sph \sma \Map_*(\Delta^k_+,X^k_+)
 \ra \prod_k F(\Delta^k_+,\Sph \sma X^k_+)
\end{equation}
where $F(A,E)$ denotes the mapping spectrum or cotensor of a space $A$ with an orthogonal spectrum $E$.
On each spectrum level, the map \eqref{eq:lx_tot_interchange} is a bijection on the underlying sets, but it is likely not a homeomorphism, because assembly maps of the form $A \sma \Map_*(B,C) \to \Map_*(B,A \sma C)$ fail to be closed inclusions (\cite{lewis1978stable}, Appendix A, 8.6).
It does not really matter, because the cocyclic spectrum $\Sigma^\infty_+ X^{\bullet + 1}$ is not Reedy fibrant, and so it must be replaced if the totalization is to be homotopically meaningful.
Taking a Reedy fibrant replacement $R\Sigma^\infty_+ X^{\bullet + 1}$ and totalizing gives a derived version of the interchange map
\[ \Sigma^\infty_+ LX \ra \Tot(\Sigma^\infty_+ X^{\bullet + 1}) \ra \Tot(R\Sigma^\infty_+ X^{\bullet + 1}). \]

\begin{prop}\label{LX_is_derived_tot}
This composite is a stable equivalence when $X$ is simply-connected.
\end{prop}



\begin{proof}
We first recall that the case where $X$ is finite follows from \cite[6.6]{kuhn2004mccord}, with $K = S^1$ and $Z = X$. To see why, we observe that the cyclic bar construction on the dual $DX$ can be made into a Reedy cofibrant cyclic spectrum (see section \ref{sec:bar}). Applying $F(-,f\Sph)$, where $f\Sph$ is a fibrant replacement of the sphere spectrum, gives a Reedy fibrant cosimplicial spectrum replacing $\Sigma^\infty_+ X^{\bullet + 1}$. One then checks that the map of Kuhn's theorem lines up with the interchange we described above.

To get the general case, it suffices to show that both sides of the interchange map commute with filtered homotopy colimits of simply-connected spaces. Using the ``cube of retracts'' terminology from \cite{malkiewich2015tower}, we identify the fibers of the coskeletal filtration of $\Tot(R\Sigma^\infty_+ X^{\bullet + 1})$ as
\[ F(\Delta^n/\partial \Delta^n, \Sigma^\infty_+ X \sma \Sigma^\infty X^{\sma n}) \simeq \Omega^n \Sigma^\infty X^{\sma n} \vee \Omega^{n} \Sigma^\infty X^{\sma (n+1)}. \]
The connectivity of these fibers tends to infinity when $X$ is simply-connected, and it follows easily that the limit of the tower commutes with such filtered homotopy colimits.

An alternative argument uses the ``cyclic coskeletal filtration'' for the right-hand side, whose fibers are
\[ F^{C_{n+1}}(\Lambda^n/\partial \Lambda^n, \Sigma^\infty X^{\sma (n+1)}) \simeq \Omega^n \Sigma^\infty X^{\sma (n+1)} \vee \Omega^{n+1} \Sigma^\infty X^{\sma (n+1)}. \]
Along the interchange map, this filtration can be shown to agree with Arone's model of the Taylor tower for $\Sigma^\infty_+ LX$ from \cite{arone_snaith}.

\end{proof}

\section{Orthogonal $G$-spectra, equivariant smash powers, and rigidity.}\label{sec:orthogonal}

We will now review the theory of orthogonal $G$-spectra and prove our rigidity theorem for the geometric fixed point functor $\Phi^G$.
This result is a technical linchpin that underlies the rest of our treatment of cyclotomic spectra and the cyclic bar construction.
It allows us to cleanly reconstruct and extend the model of $THH$ presented in \cite{angeltveit2014relative}.

\subsection{Basic definitions, model structures, and fixed points.}\label{gspectra_basics}

We take these definitions from \cite{mandell2002equivariant} and \cite{hhr}.
\begin{df}
If $G$ is a fixed compact Lie group, an \emph{orthogonal $G$-spectrum} is a sequence of based spaces $\{X_n\}_{n=0}^\infty$ equipped with
\begin{itemize}
\item A continuous action of $G \times O(n)$ on $X_n$ for each $n$
\item A $G$-equivariant structure map $\Sigma X_n \to X_{n+1}$ for each $n$
\end{itemize}
such that the composite
\[ S^p \sma X_n \ra \ldots \ra S^1 \sma X_{(p-1)+n} \ra X_{p + n} \]
is $O(p) \times O(n)$-equivariant. A map of orthogonal $G$-spectra $X \to Y$ is a collection of maps $X_n \to Y_n$ commuting with all the structure, including the $G$-actions.
\end{df}
\begin{df}
Let $U$ be a complete $G$-universe as in \cite{mandell2002equivariant}.
The category $\mathscr J_G$ has objects the finite-dimensional $G$-representations $V \subset U$, or any orthogonal $G$-representation isomorphic to such a subspace. The mapping spaces $\mathscr J_G(V,W)$ are the Thom spaces $O(V,W)^{W - V}$, consisting of linear isometries $f: V \to W$ with choices of point in the orthogonal complement $W - f(V)$.
The group $G$ acts on $O(V,W)^{W - V}$ by conjugating the map and acting on the point in $W - f(V)$.
\end{df}
\begin{df}
A \emph{$\mathscr J_G$-space} is an equivariant functor $\mathscr J_G$ into based $G$-spaces and nonequivariant maps.
That is, each $V$ is assigned to a based space $X(V)$, and for each pair $V,W$ the map
\[ \mathscr J_G(V,W) \ra \Map_*(X(V),X(W)) \]
is equivariant. A map of $\mathscr J_G$-spaces is a collection of $G$-equivariant maps $X(V) \to Y(V)$ commuting with the action of $\mathscr J_G$.
\end{df}
\begin{prop}
Every $\mathscr J_G$-space gives an orthogonal $G$-spectrum by restricting to $V = \R^n$; denote this functor $\mc I_U^{\R^\infty}$.
Conversely, given an orthogonal $G$-spectrum $X$ one may define a $\mathscr J_G$-space by the rule
\[ X(V) = X_n \sma_{O(n)} O(\R^n,V)_+, \quad n = \dim V, \]
with $G$ acting diagonally on $X_n$ and on $O(\R^n,V) = \mathscr J_G(\R^n,V)$.
Denote this functor $\mc I_{\R^\infty}^U$.
Then $\mc I_{\R^\infty}^U$ and $\mc I_U^{\R^\infty}$ are inverse equivalences of categories.
\end{prop}
\begin{df}
Given a $G$-representation $V$ and based $G$-space $A$, the \emph{free spectrum} $F_V A$ is the $\mathscr J_G$-space
\[ (F_V A)(W) := \mathscr J_G(V,W) \sma A \]
For fixed $V$, the functor $A \mapsto F_V A$ is the left adjoint to the functor that evaluates a $\mathscr J_G$-space at $V$.
\end{df}

\begin{prop} \cite{mandell2002equivariant} \label{prop:model_structure}
There is a cofibrantly generated model structure on the category of orthogonal $G$-spectra, in which the cofibrations are the retracts of the cell complexes built from
\[ F_V((G/H \times \partial D^k)_+) \hookrightarrow F_V((G/H \times D^k)_+) \qquad k \geq 0, \ H \leq G, \ V \subset U \]
and the weak equivalences are the maps inducing isomorphisms on the stable homotopy groups
\[ \pi_k^H(X) = \left\{ \begin{array}{rl}
\underset{V \subset U}\colim\, \pi_k(\Map^H_*(S^V,X(V))) , & \qquad k \geq 0 \\
\underset{V \subset U}\colim\, \pi_0(\Map^H_*(S^{V-\R^{|k|}},X(V))) , & \qquad k < 0, \ \R^{|k|} \subset V \end{array}\right., \]
where $\Map^H_*(-,-)$ denotes the space of $H$-equivariant maps.
\end{prop}

\begin{prop} \cite{mandell2002equivariant}
The category $\mathscr J$ is symmetric monoidal, using the direct sum of representations. The Day convolution along $\mathscr J$ defines a smash product on the category of orthogonal $G$-spectra which makes it into a closed symmetric monoidal category. This smash product is a left Quillen bifunctor with respect to the above model structure.
\end{prop}

When working with $G = S^1$, it is common to consider a broader class of weak equivalences that see only the finite subgroups $C_n \leq S^1$.
\begin{df}\label{df:f_equiv}
A map of $S^1$-spectra is an \emph{$\mathcal F$-equivalence} if it is an equivalence as a map of $C_n$-spectra for all $n \geq 1$; equivalently it is an isomorphism on the homotopy groups $\pi_k^{C_n}(X)$ for all $n \geq 1$.
\end{df}

Next we recall the definitions of genuine and geometric fixed points.
If $X$ is a $G$-space and $H \leq G$ is a subgroup, the fixed point subspace $X^H$ has a natural action by only the normalizer $NH \leq G$.
Of course $H$ acts trivially and so we are left with a natural action by the \emph{Weyl group}
\[ WH = NH/H \cong \textup{Aut}_G(G/H) \]
When $X$ is a $G$-spectrum there are two natural notions of $H$-fixed points, each of which gives a $WH$-spectrum:

\begin{df}
For a $\mathscr J_G$-space $X$ and a subgroup $H \leq G$, the \emph{categorical fixed points} $X^H$ are the $\mathscr J_{WH}$-space which on each $H$-fixed $G$-representation $V \subset U^H \subset U$ is just the fixed points $X(V)^H$.
More simply, if $X$ is an orthogonal $G$-spectrum then $X^H$ is obtained by taking $H$-fixed points levelwise.
\end{df}

\begin{prop}
The categorical fixed points are a Quillen right adjoint from $G$-spectra to $WH$-spectra.
Their right-derived functor is called the genuine fixed points.
\end{prop}

\begin{df}\label{df:geometric_fixed_points}
If $X$ is a $\mathscr J_G$-space and $H \leq G$ then the \emph{geometric fixed points} $\Phi^{H} X$ are defined as the coequalizer
\[ \bigvee_{V,W} F_{W^{H}} S^0 \sma \mathscr J_G^{H}(V,W) \sma X(V)^{H} \rightrightarrows \bigvee_V F_{V^{H}} S^0 \sma X(V)^{H}  \ra \Phi^{H} X \]
These are naturally $\mathscr J_{WH}$-spaces on the complete $WH$-universe $U^H$.
\end{df}

\begin{thm}
The geometric fixed points $\Phi^H$ satisfy the following technical properties:
\begin{enumerate}
\item There is a natural isomorphism of $WH$-spectra
\[ \Phi^H F_V A \cong F_{V^H} A^H \]
\item $\Phi^H$ commutes with all coproducts, pushouts along a levelwise closed inclusion, and filtered colimits along levelwise closed inclusions.
\item $\Phi^H$ preserves all cofibrations, acyclic cofibrations, and weak equivalences between cofibrant objects.
\item If $H \leq K \leq G$ then $\Phi^H$ commutes with the change-of-groups from $G$ down to $K$.
\item There is a canonical commutation map
\[ \Phi^G(X \sma Y) \overset\alpha\ra \Phi^G X \sma \Phi^G Y \]
which is an isomorphism when $X$ or $Y$ is cofibrant \cite[A.1]{blumberg2013homotopy}.
\end{enumerate}
\end{thm}

\begin{rmk}
The geometric fixed point functor $\Phi^H$ is not a left adjoint, since it does not commute with all colimits. A simple counterexample with $G = \Z/2$ is given by the suspension spectra of the diagram of spaces
\[ \xymatrix @R=2em @C=2em{
(\Z/2)_+ \ar[r] \ar[d] & ({*})_+ \\
({*})_+ } \]
Therefore $\Phi^H$ is not a Quillen left adjoint. However, since it still preserves weak equivalences between cofibrant orthogonal $G$-spectra, we define a \emph{left-derived geometric fixed point functor} $X \leadsto \Phi^H(cX)$ by composing $\Phi^H$ with a cofibrant replacement functor $c$ in the above model structure.
\end{rmk}

It will be important for us that these derived geometric fixed points measure the weak equivalences of $G$-spectra. This is standard; for instance we can deduce it from \cite[2.52]{hhr} and \cite[3.5(vi), 4.12]{mandell2002equivariant}.

\begin{prop}\label{prop:geometric_fixed_pts_measure_equivs}
A map $X \to Y$ of orthogonal $G$-spectra is a weak equivalence if and only if the induced map of derived geometric fixed points $\Phi^H(cX) \to \Phi^H(cY)$ is an equivalence of spectra for all $H \leq G$.
\end{prop}

As a result, a map of $S^1$-spectra $X \to Y$ is an $\mathcal F$-equivalence if and only if $\Phi^{C_n}(cX) \to \Phi^{C_n}(cY)$ is an equivalence for all $n \geq 1$.

Finally, though it does not seem to appear in the literature, the iterated fixed points map of \cite{blumberg2013homotopy} easily generalizes:
\begin{prop}
If $H \leq K \leq NH \leq G$ then there is a natural iterated fixed points map
\[ \Phi^K X \overset{\textup{it}}\ra \Phi^{K/H} \Phi^H X \]
which is an isomorphism when $X = F_V A$, and therefore an isomorphism on all cofibrant spectra.
When $H$ and $K$ are normal this is a map of $G/K$-spectra.
\end{prop}

\subsection{The Hill-Hopkins-Ravenel norm isomorphism.}

When $X$ is an orthogonal spectrum, the smash product $X^{\sma n}$ has an action of $C_n \cong \Z/n$ which rotates the factors.
This makes $X^{\sma n}$ into an orthogonal $C_n$-spectrum.
It is natural to guess that the geometric fixed points of this $C_n$-action should be $X$ itself, and in fact there is natural diagonal map
\[ X \overset\Delta\ra \Phi^{C_n} X^{\sma n} \]
When $X$ is cofibrant, this map is an \emph{isomorphism}.
More generally, if $G$ is a finite group, $H \leq G$, and $X$ is an orthogonal $H$-spectrum, we can define a smash product of copies of $X$ indexed by $G$
\[ N^G_H X := \bigwedge_{g_iH \in G/H} (g_iH)_+ \sma_H X \cong \bigwedge^{|G/H|} X \]
This construction is the \emph{multiplicative norm} defined by Hill, Hopkins, and Ravenel.
This can be given a $G$-action, which depends on some fixed choice of representatives $g_i H$ for each left coset of $H$ (cf. \cite{bohmann2014comparison}, \cite{hhr}).
Changing the choice of representatives changes this action, but only up to natural isomorphism.
We therefore implicitly assume that such representatives have been chosen.
The general form of the above observation about $X^{\sma n}$ is then
\begin{thm}\cite[B.209]{hhr}
There is a natural ``diagonal'' map of orthogonal spectra
\[ \Phi^H X \overset\Delta\ra \Phi^G N^G_H X \]
When $X$ is cofibrant, $\Delta$ is an isomorphism.
\end{thm}

The full proof now appears in \cite{hhr}, but for the reader's convenience we also summarize the proof below.
\begin{proof}
If $A$ is just a based $H$-space, the indexed smash product of $A$ over $G/H$ has fixed points $A^H$:
\[ A^H \congar \left( N^G_H A \right)^G \cong \left( \bigwedge^{|G/H|} A \right)^G \]
Here the map from left to right is the diagonal,
\[ a \in A^H \mapsto (a,\ldots,a) \]

Now suppose $X$ is an orthogonal $H$-spectrum. We start by taking its coequalizer presentation
\[ \bigvee_{V,W} F_W S^0 \sma \mathscr J_H(V,W) \sma X(V) \rightrightarrows \bigvee_V F_V S^0 \sma X(V) \ra X \]
and taking $\Phi^G N^G_H$ of everything in sight.
Since $\Phi^G N^G_H$ commutes with wedges and smashes up to isomorphism, this gives
\[ \bigvee_{V,W} \Phi^G N^G_H F_W S^0 \sma (N^G_H \mathscr J_H(V,W))^G \sma (N^G_H X(V))^G \rightrightarrows \bigvee_V \Phi^G N^G_H F_V S^0 \sma (N^G_H X(V))^G \]
\[ \ra \Phi^G N^G_H X \]
which simplifies to
\[ \bigvee_{V,W} \Phi^G N^G_H F_W S^0 \sma \mathscr J_H^H(V,W) \sma X(V)^H \rightrightarrows \bigvee_V \Phi^G N^G_H F_V S^0 \sma X(V)^H \ra \Phi^G N^G_H X \]
As a diagram, this is no longer guaranteed to be a coequalizer system, but it still commutes.
We can simplify using the string of isomorphisms
\begin{eqnarray*}
\Phi^G N^G_H F_V A &\cong & \Phi^G F_{\textup{Ind}^G_H V} (N^G_H A) \\
 &\cong & F_{(\textup{Ind}^G_H V)^G} (N^G_H A)^G \\
 &\cong & F_{V^H} A^H
\end{eqnarray*}
for any based $H$-space $A$ and $H$-representation $V$.
This gives
\[ \bigvee_{V,W} F_{W^H} S^0 \sma \mathscr J_H^H(V,W) \sma X(V)^H \rightrightarrows \bigvee_V F_{V^H} S^0 \sma X(V)^H \ra \Phi^G N^G_H X \]
and the coequalizer of the first two terms is exactly $\Phi^H X$.
The universal property of the coequalizer then gives us a map
\[ \Phi^H X \ra \Phi^G N^G_H X \]
and we take this as the definition of the diagonal map.

Now consider the special case when $X = F_V A$.
The inclusion of the term
\[ F_{V^H} S^0 \sma A^H \]
into the above coequalizer system maps forward isomorphically to $\Phi^H X$, and so we can evaluate the diagonal map by just examining this term.
But back at the top of our proof, the inclusion of the term
\[ \Phi^G N^G_H F_V S^0 \sma (N^G_H A)^G \]
also maps forward isomorphically to $\Phi^G N^G_H X$.
Therefore up to isomorphism, the diagonal map becomes the string of maps we used to connect $F_{V^H} S^0 \sma A^H$ to $\Phi^G N^G_H F_V S^0 \sma (N^G_H A)^G$, but these maps were all isomorphisms.
Therefore the diagonal is an isomorphism when $X = F_V A$.
It is straightforward to verify that both sides preserve coproducts, pushouts along $h$-cofibrations, and sequential colimits along $h$-cofibrations, so by induction, the diagonal is an isomorphism for all cofibrant $X$.
\end{proof}

\subsection{A rigidity theorem for geometric fixed points}\label{rigidity_section}

Let $G$ be a compact Lie group. We will prove that the geometric fixed point functor $\Phi^G$ is \emph{rigid}, in the sense that it admits very few point-set level natural transformations into other functors. Let $G\cat{Sp}^O$ denote the category of orthogonal $G$-spectra and $G$-equivariant maps between them. Let $\cat{Free}$ be the full subcategory on the free spectra $F_V A$, for all $G$-representations $V$ and based $G$-spaces $A$.
Let
\[ \sma \circ (\Phi^G,\ldots,\Phi^G) : \prod^k \cat{Free} \ra \cat{Sp}^O \]
denote the composite of the geometric fixed points and the $k$-fold smash product, with $k \geq 1$.

\begin{prop}
The only endomorphisms of $\sma \circ (\Phi^G)^k$ are zero and the identity.
\end{prop}
\begin{proof}
Consider a natural transformation $T: \sma \circ (\Phi^G)^k \to \sma \circ (\Phi^G)^k$.
On $(F_0 S^0,F_0 S^0, \ldots, F_0 S^0)$, $T$ gives a map of spectra
\[ F_0 S^0 \ra F_0 S^0 \]
which is determined by at level 0 a choice of point in $S^0$.
So there are only two such maps, the identity and zero.

Assume that $T$ is the identity on this object.
Then consider $T$ on $(F_{V_1} S^0,F_{V_2} S^0, \ldots, F_{V_k} S^0)$:
\[ F_{V_1^G \oplus V_2^G \oplus \ldots \oplus V_k^G} S^0 \ra F_{V_1^G \oplus V_2^G \oplus \ldots \oplus V_k^G} S^0 \]
Let $m_i := \dim V_i^G$ and fix an isomorphism between $\R^{m_i}$ and $V_i^G$.
The above map is determined by what it does at level $m_1 + \ldots + m_k$:
\[ O(m_1 + \ldots + m_k)_+ \ra O(m_1 + \ldots + m_k)_+ \]
This map, in turn, is determined by the image of the identity point, which is some element $P \in O(m_1 + \ldots + m_k)_+$.
Now for any point $(t_1,\ldots,t_k) \in S^{m_1} \sma \ldots \sma S^{m_k}$ we can choose maps of spectra $F_{V_i} S^0 \to F_0 S^0$ which at level $V_i$ send the non-basepoint of $S^0$ to the point $t_i \in S^{m_i} \cong (S^{V_i})^G$.
Since $T$ is a natural transformation, this square commutes for all choices of $(t_1,\ldots,t_k)$:
\[ \xymatrix{
O(m_1 + \ldots + m_k)_+ \ar[r]^-{\cdot P} \ar[d]^-{\textup{ev}_{(t_1,\ldots,t_k)}} & O(m_1 + \ldots + m_k)_+ \ar[d]^-{\textup{ev}_{(t_1,\ldots,t_k)}} \\
S^{m_1 + \ldots + m_k} \ar[r]^-\id & S^{m_1 + \ldots + m_k} } \]
Since $O(m_1 + \ldots + m_k)$ acts faithfully on the sphere $S^{m_1 + \ldots + m_k}$, we must have $P = \id$. Therefore our natural transformation $T$ acts as the identity on the $k$-tuple of spectra $(F_{V_1} S^0,F_{V_2} S^0, \ldots, F_{V_k} S^0)$.

Finally let $A_1, \ldots, A_k$ be a sequence of $G$-spaces and consider $T$ on $(F_{V_1} A_1, \ldots, F_{V_k} A_k)$.
Each collection of choices of point $a_i \in A_i^G$ gives a sequence of maps $F_{V_i} S^0 \to F_{V_i} A_i$, and applying $T$ to this sequence of maps gives a commuting square
\[ \xymatrix{
F_{V_1^G \oplus \ldots \oplus V_k^G} S^0 \sma \ldots \sma S^0 \ar[r]^-\id \ar[d]^-{F_{\ldots}(a_1,\ldots,a_k)} & F_{V_1^G \oplus \ldots \oplus V_k^G} S^0 \sma \ldots \sma S^0 \ar[d]^-{F_{\ldots}(a_1,\ldots,a_k)} \\
F_{V_1^G \oplus \ldots \oplus V_k^G} A_1^G \sma \ldots \sma A_k^G \ar[r]^-{T} & F_{V_1^G \oplus \ldots \oplus V_k^G} A_1^G \sma \ldots \sma A_k^G } \]
From inspection of level $m_1 + \ldots + m_k$, the bottom map must be the identity on the point $\id \sma (a_1,\ldots,a_k)$.
But this is true for all $(a_1,\ldots,a_k)$ and so the bottom map is the identity.
Therefore $T$ is the identity on $(F_{V_1} A_1, \ldots, F_{V_k} A_k)$, so it is the identity on every object in $\prod^k \cat{Free}$.

For the second case, we assume $T$ is zero on $(F_0 S^0,\ldots, F_0 S^0)$ and follow the same steps as before, concluding that $T$ is zero on $(F_{V_1} S^0, \ldots, F_{V_k} S^0)$ and then it is zero on $(F_{V_1} A_1, \ldots, F_{V_k} A_k)$.
\end{proof}

To derive corollaries, we say that a functor $\phi: \prod^k G\cat{Sp}^O \to \cat{Sp}^O$ is \emph{rigid} if restricting to the subcategory $\prod^k \cat{Free}$ gives an injective map on natural transformations out of $\phi$.
In other words, a natural transformation out of $\phi$ is determined by its behavior on the subcategory $\cat{Free}$.
\begin{cor}
If $\phi_1$ and $\phi_2$ are functors $\prod^k G\cat{Sp}^O \to \cat{Sp}^O$ which when restricted to the subcategory $\prod^k \cat{Free}$ are separately isomorphic to $\sma \circ (\Phi^G)^k$, and $\phi_1$ is rigid, then there is at most one nonzero natural transformation $\phi_1 \to \phi_2$.
\end{cor}

The example we are interested in is the smash product of geometric fixed points.
\begin{prop}
$\sma \circ (\Phi^G,\ldots,\Phi^G)$ is a rigid functor.
\end{prop}
\begin{proof}
For any orthogonal $G$-spectrum $X$, let $\xi$ denote the map
\[ \xi: \bigvee_{V \subset U} F_V X(V) \ra X \]
whose $V$th summand is adjoint to the identity map of $X(V)$. It suffices to show that $\sma \circ (\Phi^G,\ldots,\Phi^G)$ takes $(\xi,\ldots,\xi)$ to a map of orthogonal spectra that is surjective on every spectrum level. We will describe this in detail in the case of $k = 2$, $(X,Y) \leadsto \Phi^G X \sma \Phi^G Y$, but the other cases are similar.

The smash product commutes with colimits in each variable, and this gives a definition of $\Phi^G X \sma \Phi^G Y$ as a colimit of a diagram with four terms. We rearrange this into a single coequalizer diagram and conclude that there is a natural levelwise surjection of spectra 
\[ \bigvee_{V',W' \subset U} F_{V'^G} X(V')^G \sma F_{W'^G} Y(W')^G \ra \Phi^G X \sma \Phi^G Y \]
for all orthogonal $G$-spectra $X$ and $Y$. Applying this construction to $(\xi, \xi)$ gives a commuting square
\[ \xymatrix @C=0em{
\bigvee_{V,W \subset U} \Phi^G F_{V}X(V) \sma \Phi^G F_{W}Y(W) \ar[r]^-{\Phi^G (\xi) \sma \Phi^G(\xi)} & \Phi^G X \sma \Phi^G Y \\
Z \ar@{->>}[u] \ar[r] & \bigvee_{V',W' \subset U} F_{V'^G} X(V')^G \sma F_{W'^G} Y(W')^G \ar@{->>}[u] 
} \]
\[ Z = \bigvee_{V',W',V,W \subset U} F_{V'^G} [\mathscr J_G(V,V') \sma X(V)]^G \sma F_{W'^G} [\mathscr J_G(W,W') \sma Y(W)]^G \]
in which the vertical maps are levelwise surjections. We wish to show that $\Phi^G (\xi) \sma \Phi^G(\xi)$ is surjective, and for this it suffices to show that the bottom horizontal map is surjective. This follows by examining the summands where $V = V'$ and $W = W'$, and noting that the action map $O(V)_+ \sma X(V) \to X(V)$ is surjective on the $G$-fixed points. (Alternatively, one can show that the top horizontal and right vertical maps may be identified by a homeomorphism.)
\end{proof}


As a result, we get new rigidity statements for the maps relating geometric fixed points and smash powers:
\begin{thm}
Let $X$ and $Y$ denote arbitrary $G$-spectra.
Then the commutation map
\[ \Phi^G X \sma \Phi^G Y \overset{\alpha}\ra \Phi^G (X \sma Y) \]
is the only nonzero natural transformation from $\Phi^G X \sma \Phi^G Y$ to $\Phi^G (X \sma Y)$.
\end{thm}

\begin{rmk}
If $X$ and $Y$ are $G$-spectra and $H \leq G$ then there is more than one natural map
\[ \Phi^H X \sma \Phi^H Y \ra \Phi^H (X \sma Y). \]
Indeed, we could take any element $g$ in the center $Z(G)$, and post-compose $\alpha_H$ with the map $\mc I_{\R^\infty}^U g$ that acts on the trivial-representation levels by the action of $g$.
However $\alpha_H$ is the only natural transformation that respects the forgetful functor to $H$-spectra.
In other words, it is the only one that is natural with respect to all of the $H$-equivariant maps of spectra, and not just the $G$-equivariant ones.
Similar considerations apply to the iterated fixed points map below.
\end{rmk}

\begin{thm}
Let $G$ be a finite group, and let $X$ denote an arbitrary $H$-spectrum with $H \leq G$.
Then the Hill-Hopkins-Ravenel diagonal map
\[ \Phi^H X \overset{\Delta}\ra \Phi^G N^G_H X \]
is the only such map that is both natural and nonzero.
\end{thm}

\begin{thm}
If $X$ is a $G$-spectrum and $N \leq G$ is a normal subgroup, then the iterated fixed points map
\[ \Phi^G X \overset{\textup{it}}\ra \Phi^{G/N} \Phi^N X \]
is characterized by the property that it is natural in $X$ and nonzero.
\end{thm}

We end with five more corollaries, which served as the motivation for the rigidity result.
The first corollary is the most important for our work on tensors and duals of cyclotomic spectra.
\begin{prop}\label{smash_iterated_commute}
If $X$ and $Y$ are a $G$-spectra and $N \leq G$ is a normal subgroup, then the following rectangle commutes:
\[ \xymatrix @C=4em{
\Phi^G X \sma \Phi^G Y \ar[d]^-{\textup{it} \sma \textup{it}} \ar[rr]^-{\alpha_G}
 && \Phi^G (X \sma Y) \ar[d]^-{\textup{it}} \\
\Phi^{G/N} \Phi^N X \sma \Phi^{G/N} \Phi^N Y \ar[r]^-{\alpha_{G/N}}
 & \Phi^{G/N} (\Phi^N X \sma \Phi^N Y) \ar[r]^-{\Phi^{G/N} \alpha_N}
 & \Phi^{G/N} \Phi^N (X \sma Y)
} \]
\end{prop}

The next two corollaries help us simplify and clarify the theory of cyclic orthogonal spectra.
\begin{prop}\label{g_acts_trivially} (cf. \cite[Lem 4.5]{angeltveit2014relative})
If $X$ is a $G$-spectrum and $g \in Z(G)$, then multiplication by $g$ on the trivial representation levels gives a map of $\mathscr J_G$-spaces
\[ \xymatrix @C=4em{ X \ar[r]^-{\mc I_{\R^\infty}^U g} & X } \]
which on fixed points
\[ \xymatrix @C=5em{ \Phi^G X \ar[r]^-{\Phi^G \mc I_{\R^\infty}^U g} & \Phi^G X } \]
is the identity map.
\end{prop}

\begin{prop}\label{rotating_compatibility}
If $X$ and $Y$ are orthogonal spectra, then the self-map of orthogonal $C_r$-spectra
\[ f: N^{C_r} (X \sma Y) \cong X^{\sma r} \sma Y^{\sma r} \ra X^{\sma r} \sma Y^{\sma r} \]
which rotates only the $Y$ factors but not the $X$ factors fits into a commuting triangle
\[ \xymatrix @R=0.8em{
& \Phi^{C_r} (X^{\sma r} \sma Y^{\sma r}) \ar[dd]^-{\Phi^{C_r} \mc I_{\R^\infty}^U f} \\
X \sma Y \ar[ru]^-\Delta \ar[rd]_-\Delta & \\
& \Phi^{C_r} (X^{\sma r} \sma Y^{\sma r}) } \]
\end{prop}

The next corollary requires more explanation.
Let $X$ be an orthogonal spectrum, and consider the diagonal map
\[ X^{\sma m} \overset{\Delta_n}\ra \Phi^{C_n} (X^{\sma m})^{\sma n} \]
If we write $(X^{\sma m})^{\sma n}$ in lexicographical order
\[ (X^{\sma m}) \sma (X^{\sma m}) \sma \ldots \sma (X^{\sma m}) \]
then there is an obvious $C_{mn}$-action which rotates the terms.
This commutes with the action of the subgroup $C_n$, so it passes to a $C_{mn}$-action on the geometric fixed points.
By Prop \ref{g_acts_trivially}, the subgroup $C_n$ acts trivially, giving a $C_m$-action on the fixed points.
\begin{prop}\label{cmn_is_cmcn}
Under these conventions, $\Delta_n$ is $C_m$-equivariant.
\end{prop}

\begin{proof}
Let $g$ denote the generator of $C_m$, and $h$ the generator of $C_{mn}$.
Since the diagonal is natural, $\Delta_n$ is equivariant with respect to the action of $g$, but with $g$ acting on $(X^{\sma m})^{\sma n}$ by rotating each $X^{\sma m}$ separately.
If we apply $g$ and then the inverse of $h$, the composite matches the description of the map $f$ of Prop \ref{rotating_compatibility}.
Therefore $f^{-1} \circ \Delta_n = \Delta_n$, so
\[ \Delta_n \circ g = g \circ \Delta_n = g \circ f^{-1} \circ \Delta_n = h \circ \Delta_n \]
Therefore $\Delta_n$ is $C_m$-equivariant.
\end{proof}

\begin{rmk}
This argument generalizes: the diagonal map $\Delta_n$ commutes with any automorphism of $X^{\sma mn}$ coming from a self-map of the $C_n$-set $C_m \times C_n$ that gives the identity on the quotient set $C_m$. In particular, $C_{mn}$ may be identified with $C_m \times C_n$ as $C_n$-sets with quotient $C_m$.
\end{rmk}

Our final corollary will be the key ingredient for showing that the cyclotomic structure maps on the cyclic bar construction are compatible with each other.
\begin{prop}\label{normdiag_iterated_commute}
If $X$ is an ordinary spectrum and $m,n \geq 0$ then the following square commutes:
\[ \xymatrix @C=5em{
X \ar[r]^-{\Delta_{C_{mn}}} \ar[d]^-{\Delta_{C_m}} & \Phi^{C_{mn}} X^{\sma mn} \ar[d]^-{\textup{it}} \\
\Phi^{C_m} X^{\sma m} \ar[r]^-{\Phi^{C_m}(\Delta_n)} & \Phi^{C_m} \Phi^{C_n} X^{\sma mn}
} \]
\end{prop}

\begin{rmk}
It is reasonable to expect that $\Delta_n$ coincides with the generalized HHR diagonal
\[ N^{C_{mn}/C_n} X \overset{\Delta_*}\ra \Phi^{C_n} N^{C_{mn}} X \]
of \cite[2.19]{angeltveit2014relative}. Of course the above proposition is true for $\Delta_*$ as well.
\end{rmk}

\section{Cyclic orthogonal spectra and the cyclic bar construction.}\label{sec:bar}

Now we will integrate the modern technology from section \ref{sec:orthogonal} into the classical theory from section \ref{sec:cyclic}. We prove a few more properties of cyclic and cocyclic orthogonal spectra that concern the genuinely equivariant structure. Then we describe the construction and properties of the cyclic bar construction in orthogonal spectra, expanding on the treatment in \cite{angeltveit2014relative}.

\subsection{Equivariant properties of cyclic and cocyclic spectra.}

Let $X_\bullet$ be a cyclic orthogonal spectrum.
Then $\sd_r X_\bullet$ is an $r$-cyclic orthogonal spectrum.
At each simplicial level, $(\sd_r X)_{n-1}$ is an orthogonal spectrum with $C_r$-action generated by the $n$th power of the cycle map $t_{rn-1}^n$.
This commutes with all the face, degeneracy, and cycle maps, making $\sd_r X_\bullet$ an $r$-cyclic object in orthogonal $C_r$-spectra.
So we may take the geometric fixed points on each level separately.
\begin{prop}\label{cyclic_spectra_realization}
If $X_\bullet$ is a cyclic spectrum then $\Phi^{C_r} \sd_r X_\bullet$ is naturally a cyclic spectrum, and there is a natural $S^1$-equivariant isomorphism
\[ |\Phi^{C_r} \sd_r X_\bullet| \cong \rho_r^* \Phi^{C_r} |X_\bullet| \]
\end{prop}
\begin{proof}
Since geometric fixed points is a functor, we know that $\Phi^{C_r} \sd_r X_\bullet$ is at least an $r$-cyclic orthogonal spectrum.
By Prop \ref{g_acts_trivially}, the $n$th power of the cycle map $t_{rn-1}^n$ acts trivially on the geometric fixed points. Therefore $\Phi^{C_r} \sd_r X_\bullet$ is actually a cyclic spectrum, i.e. it factors in a canonical way through the quotient functor $P_r: \mathbf\Lambda_r \to \mathbf\Lambda$.

Using $P_r \Phi^{C_r} \sd_r X_\bullet$ to denote $\Phi^{C_r} \sd_r X_\bullet$ as an $r$-cyclic spectrum, we have the equivariant isomorphisms
\[ |\Phi^{C_r} \sd_r X_\bullet| \cong \rho_r^* |P_r \Phi^{C_r} \sd_r X_\bullet| \cong \rho_r^* \Phi^{C_r} |\sd_r X_\bullet| \cong \rho_r^* \Phi^{C_r} |X_\bullet| \]
where the middle map is the canonical commutation of $\Phi^{C_r}$ with geometric realization.
These are obtained from the maps of Prop \ref{cyclic_fixed_points} applied to the term $F_{V^{C_r}} S^0 \sma X(V)^{C_r}$ in the coequalizer system for $\Phi^{C_r} X$. They pass to the coequalizer because $\rho_r^*$, $P_r$, $\sd_r$, and geometric realization all commute with colimits.
\end{proof}

We already know (Prop \ref{reedy_cofibrant_spectra}) that the realization functor $|X_\bullet|$ preserves weak equivalences when $X_\bullet$ is Reedy cofibrant.
We will also need to know when $|X_\bullet|$ is cofibrant.
\begin{prop}\label{reedy_cof_implies_realization_is_cof_s1_spectra}
If $X_\bullet$ is a cyclic spectrum, $X_{-1}$ is a cofibrant spectrum, and each cyclic latching map $L_n^\cyc X \to X_n$ is a cofibration of $C_{n+1}$-spectra, then $|X_\bullet|$ is a cofibrant $S^1$-spectrum.
\end{prop}
\begin{proof}
As in Prop \ref{reedy_cof_implies_realization_is_cof_s1_spaces}, we we reduce to checking that the $C_{n+1}$ orbits of a pushout-product of a $C_{n+1}$-cell of spectra and a free $S^1$-cell of spaces is an $S^1$-cofibration:
\[ [ (F_V (C_{n+1}/C_r \times \partial D^k)_+ \ra F_V (C_{n+1}/C_r \times D^k)_+) \square (S^1 \times \partial D^\ell \ra S^1 \times D^\ell)_+ ]_{C_{n+1}} \]
Here $V$ is any finite-dimensional $C_r$-representation. This simplifies to
\[ [ F_V (C_{n+1}/C_r)_+ \sma_{C_{n+1}} S^1_+ ] \sma (\partial D^{k + \ell} \ra D^{k + \ell})_+ \]
It suffices to show the left-hand term is cofibrant as an $S^1$-spectrum, but it is obtained by applying the left Quillen functor $- \sma_{C_{n+1}} S^1_+$ to the $C_{n+1}$-cofibrant object $F_V (C_{n+1}/C_r)_+$, so it is cofibrant.
\end{proof}

Next, let $X^\bullet$ be a cocyclic orthogonal spectrum.
Then $\sd_r X^\bullet$ is an $r$-cocyclic orthogonal spectrum, and by the same argument as above, $\Phi^{C_r} \sd_r X^\bullet$ is naturally a cocyclic orthogonal spectrum.
As before, we get the string of equivariant maps
\[ \Tot(\Phi^{C_r} \sd_r X^\bullet) \cong \rho_r^* \Tot(P_r \Phi^{C_r} \sd_r X^\bullet) \la \rho_r^* \Phi^{C_r} \Tot(\sd_r X^\bullet) \cong \rho_r^* \Phi^{C_r} \Tot(X^\bullet) \]
The middle map is the canonical commutation of $\Phi^{C_r}$ with totalization, but as one might expect, it is not an isomorphism.
\begin{prop}\label{tot_phi_interchange}
There is a natural interchange map
\[ \Phi^{C_r} \Tot(Z^\bullet) \ra \Tot(\Phi^{C_r} Z^\bullet) \]
for cosimplicial spectra with $C_r$-actions.
\end{prop}
\begin{proof}
The interchange map is given canonically by universal properties, using the shorthand diagram
\[ \xymatrix @R=1em @C=1em {
& \Tot \Phi^{C_r} \ar[rr] && \prod_k \Phi^{C_r} \ar@<-.5ex>[rr] \ar@<.5ex>[rr] && \prod_{k,\ell} \Phi^{C_r} \\
\Phi^{C_r} \Tot \ar@{-->}[ru] &&&&& \\
&&& \prod_k \bigvee_V \ar[uu] \ar@<-.5ex>[rr] \ar@<.5ex>[rr] && \prod_{k,\ell} \bigvee_V \ar[uu] \\
\bigvee_V \Tot \ar[uu] \ar[rr] && \bigvee_V \prod_k \ar[ru] \ar@<-.5ex>[rr]|\hole \ar@<.5ex>[rr]|\hole && \bigvee_V \prod_{k,\ell} \ar[ru] & \\
&&& \prod_k \bigvee_{V,W} \ar@<-.5ex>[uu] \ar@<.5ex>[uu] && \\
\bigvee_{V,W} \Tot \ar@<-.5ex>[uu] \ar@<.5ex>[uu] \ar[rr] && \bigvee_{V,W} \prod_k \ar[ru] \ar@<-.5ex>[uu] \ar@<.5ex>[uu] &&& \\
} \]
A diagram-chase shows this is natural with respect to maps of cosimplicial spectra $Z^\bullet \to \ti Z^\bullet$.
\end{proof}

\begin{cor}\label{cocyclic_spectra_totalization}
If $X^\bullet$ is a cocyclic spectrum then $\Phi^{C_r} \sd_r X^\bullet$ is naturally a cocyclic spectrum, and there is a natural $S^1$-equivariant map
\[ \rho_r^* \Phi^{C_r} \Tot(X^\bullet) \ra \Tot(\Phi^{C_r} \sd_r X^\bullet) \]
\end{cor}

\subsection{The cyclic bar construction.} Let $R$ be an orthogonal ring spectrum. The \emph{cyclic bar construction on $R$} is the cyclic spectrum $N^\cyc_\bullet R$ with
\[ N^\cyc_n R = R^{\sma (n+1)} = R^{\sma n} \sma \uline{R} \]
We underline the last copy of $R$ since in the simplicial structure it plays a special role.
The action of $\mathbf\Lambda$ is best visualized by taking the category $[n]$ and labeling the arrows with copies of $R$:
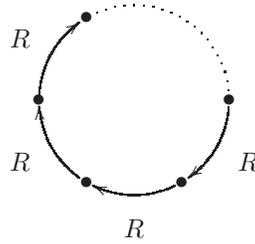
\[ \xy 0;<36pt,0pt>:
a(0)*{\bullet}="n-1";
a(120)*{\bullet}="2";
a(180)*{\bullet}="1";
a(240)*{\bullet}="0";
a(300)*{\bullet}="n";
(0,-1.35)*{R};
(1.2,-.65)*{R};
(-1.2,-.65)*{R};
(-1.2,.65)*{R};
"n-1";"n" **\crv{(0.9786,-0.565)}; ?>*\dir{>};
"n";"0" **\crv{(0,-1.13)}; ?>*\dir{>};
"0";"1" **\crv{(-0.9786,-0.565)}; ?>*\dir{>};
"1";"2" **\crv{(-0.9786,0.565)}; ?>*\dir{>};
0;a(-90) **\dir{}; (0,0)*\ellipse(1)__,=:a(120){.};
\endxy \]
Each map $[k] \rightarrow [n]$ induces a map $R^{\sma (n+1)} \rightarrow R^{\sma (k+1)}$ as follows.
Each arrow $i \rightarrow i+1$ in $[k]$ is sent to some composition $j \rightarrow \ldots \rightarrow j+\ell$ in $[n]$, which corresponds to $\ell$ copies of $R$ in $R^{\sma (n+1)}$.
We send this smash product $R^{\sma \ell}$ to the copy of $R$ in slot $i$ of $R^{\sma (k+1)}$, using the multiplication on $R$.
When $\ell = 0$, we interpret this as the unit map $\Sph \rightarrow R$.

More generally, if $\cat C$ is a category enriched in orthogonal spectra, the cyclic nerve on $\cat C$ is defined as
\[ N^\cyc_n \cat C = \bigvee_{c_0,\ldots,c_n \in \ob \cat C} \cat C(c_0,c_1) \sma \cat C(c_1,c_2) \sma \ldots \sma \cat C(c_{n-1},c_n) \sma \uline{\cat C(c_n,c_0)} \]
One may think of these objects loosely as ``functors'' from $[k]$ into $\cat C$, where ordinary products have been substituted by smash products, and this suggests the correct face, degeneracy, and cycle maps. In particular, as indicated below, the 0th face map $d_0: N^\cyc_n \cat C \to N^\cyc_{n-1} \cat C$ switches the first term $\cat C(c_0,c_1)$ past the others and composes it into $\cat C(c_n,c_0)$. The extra degeneracy map $s_{n+1}: N^\cyc_n \cat C \to N^\cyc_{n+1} \cat C$ inserts a unit $\Sph \to \cat C(c_0,c_0)$ into the underlined factor in the smash product. The cycle map $t_n: N^\cyc_n \cat C \to N^\cyc_{n} \cat C$ rotates the factors towards the right.
\[ \begin{array}{cccc}
d_0:& \cat C(c_0,c_1) \sma \cat C(c_1,c_2) \sma \ldots \sma \uline{\cat C(c_n,c_0)} &\ra &
\cat C(c_1,c_2) \sma \ldots \sma \uline{\cat C(c_n,c_1)} \\
s_{n+1}:& \ldots \sma \cat C(c_{n-1},c_n) \sma \uline{\cat C(c_n,c_0)} \sma \Sph &\ra &
\ldots \sma \cat C(c_{n-1},c_n) \sma \cat C(c_n,c_0) \sma \uline{\cat C(c_0,c_0)} \\
t_n:& \cat C(c_0,c_1) \sma \ldots \sma \cat C(c_{n-1},c_n) \sma \uline{\cat C(c_n,c_0)} &\ra & 
\cat C(c_n,c_0) \sma \cat C(c_0,c_1) \sma \ldots \sma \uline{\cat C(c_{n-1},c_n)}
\end{array} \]
If $\cat C$ has a single object, we recover the definition of $N^\cyc R$ we gave above.
\begin{df}
The \emph{topological Hochschild homology} of $\cat C$ is the geometric realization of the cyclic nerve
\[ THH(\cat C) := |N^\cyc_\bullet \cat C|. \]
\end{df}

The cyclic bar construction of orthogonal spectra is remarkable because its geometric fixed points are isomorphic to the original spectrum.
\begin{thm}\label{construction_of_THH}
If $\cat C$ is a spectral category then there are natural maps of $S^1$-spectra for $r \geq 0$
\[ \gamma_r: THH(\cat C) \ra \rho_r^* \Phi^{C_r} THH(\cat C). \]
They are compatible in the following sense: if $T = THH(\cat C)$ then the square
\[ \xymatrix @C=7em{
T \ar[r]^-{\gamma_{mn}} \ar[d]^-{\gamma_m} & \rho_{mn}^* \Phi^{C_{mn}} T \ar[d]^-{\textup{it}} \\
\rho_m^* \Phi^{C_m} T \ar[r]^-{\rho_m^*\Phi^{C_m} \gamma_n} & \rho_m^* \Phi^{C_m} \rho_n^* \Phi^{C_n} T } \]
strictly commutes. Furthermore if every $\cat C(c_i,c_j)$ is a cofibrant orthogonal spectrum, then every $\gamma_r$ is an isomorphism.
\end{thm}

\begin{rmk}
This extends one of the main results of \cite{angeltveit2014relative} from ring spectra to spectral categories.
This turns out to not be so difficult.
However the treatment in \cite{angeltveit2014relative} does not prove the above compatibility square, which seems to be harder.
Our rigidity theorem allows us to check the compatibility easily.
\end{rmk}

\begin{proof}
In essence, we need to understand the geometric fixed points of $THH(\cat C)$. We start with the isomorphism of $S^1$-spectra from Prop \ref{cyclic_spectra_realization}:
\[ |\Phi^{C_r} \sd_r N^\cyc_\bullet \cat C| \congar \rho_r^* \Phi^{C_r} |N^\cyc_\bullet \cat C| \]
It therefore suffices to understand the geometric fixed points of the subdivision $\sd_r N^\cyc_\bullet \cat C$. This is an $r$-cyclic spectrum. At simplicial level $(n-1)$ it is a wedge of smash products
\[ \bigvee_{c_0,\ldots,c_{rn-1} \in \ob \cat C} \cat C(c_0,c_1) \sma \ldots \sma \uline{\cat C(c_{rn-1},c_0)} \]
and the $C_r$-action is by $t_{rn-1}^n$, which rotates this $rn$-fold smash product by $n$ slots. In particular, the generator $\alpha \in C_r$ sends the summand $A$ indexed by $c_0,\ldots,c_{rn-1}$ to the summand $\alpha(A)$ indexed by
\[ c_{(r-1)n}, \ldots, c_{rn-1}, c_0, \ldots, c_{(r-1)n-1} \]
by a homeomorphism. The summands $A$ and $\alpha(A)$ coincide precisely when the list $c_0,\ldots,c_{rn-1}$ repeats with period $n$:
\[ c_0,c_1,\ldots,c_{n-1},c_0,c_1,\ldots,c_{n-1},c_0,c_1,\ldots,c_{n-1} \]
If this is not the case, then the $C_r$-closure $\ti A$ of $A$ does not have any levelwise $C_r$-fixed points, $\ti A(V)^{C_r} = *$. This is because any fixed point would have in its $C_r$-orbit a point $x \in A$, but then $x$ must be in the intersection $A \cap \alpha(A) = *$.

It is therefore a good idea to write $Y = \sd_r N^\cyc_{n-1} \cat C$ as the wedge of two spectra $X \vee X'$, where $X$ is the wedge of those summands $A$ such that $A = \alpha(A)$, and $X'$ contains the remaining summands. Since the levelwise fixed point functor $(-)(V)^{C_r}$ preserves wedge sums, we immediately conclude that the inclusion $X \to Y$ induces a homeomorphism on each level $X(V)^{C_r} \cong Y(V)^{C_r}$. Recalling the definition of $\Phi^{C_r}$ (Def. \ref{df:geometric_fixed_points}), we conclude that the inclusion also induces an isomorphism on the geometric fixed points $\Phi^{C_r} X \cong \Phi^{C_r} Y$.

In conclusion, the geometric fixed points of the subdivision can be rewritten as
\begin{eqnarray*}
\Phi^{C_r} \sd_r N^\cyc_{n-1} \cat C &\cong & \Phi^{C_r}\left(\bigvee_{c_0,\ldots,c_{n-1}} \left(\cat C(c_0,c_1) \sma \ldots \sma \cat C(c_{n-1},c_0)\right)^{\sma r}\right) \\
&\cong & \bigvee_{c_0,\ldots,c_{n-1}} \Phi^{C_r}\left(\cat C(c_0,c_1) \sma \ldots \sma \cat C(c_{n-1},c_0)\right)^{\sma r}. \\
\end{eqnarray*}
It remains to compare this last term to $N^\cyc_{n-1} \cat C$ using Hill-Hopkins-Ravenel norm diagonal
\[ \cat C(c_0,c_1) \sma \ldots \sma \cat C(c_{n-1},c_0) \overset{\Delta}\ra \Phi^{C_r}\left(\cat C(c_0,c_1) \sma \ldots \sma \cat C(c_{n-1},c_0)\right)^{\sma r}. \]
We want to show that these diagonal maps for each $n \geq 1$ assemble into a map of cyclic spectra
\[ N^\cyc_\bullet \cat C \overset\Delta\ra \Phi^{C_r} \sd_r N^\cyc_\bullet \cat C \]
(cf. \cite[4.6]{angeltveit2014relative}).
It easily commutes with most of the face and degeneracy maps because the diagonal is natural.
One runs into issues with $d_0$ and $t_{rn-1}$, but these are fixed by the argument we used in Prop \ref{cmn_is_cmcn}.
In brief, the $r$-fold smash $(d_0)^{\sma r}$ of $d_0$ from the cyclic structure is not the same map as $d_0$ in the $r$-cyclic structure, but they differ by the map
\[ f: \left(\cat C(c_0,c_1) \sma \ldots \sma \cat C(c_{n-1},c_0)\right)^{\sma r} \ra \left(\cat C(c_0,c_1) \sma \ldots \sma \cat C(c_{n-1},c_0)\right)^{\sma r} \]
that takes the factors $\cat C(c_{n-1},c_0)$ and cycles them while leaving all the other terms fixed.
It suffices to show that $f$ commutes with $\Delta$, but we did that in Prop \ref{rotating_compatibility}.
A similar argument works for $t_{rn-1}$.

This proves that the Hill-Hopkins-Ravenel diagonal gives a map of cyclic spectra.
We define $\gamma_r$ to be its geometric realization, combined with the $S^1$-equivariant isomorphism of Prop \ref{cyclic_spectra_realization}:
\[ |N^\cyc_\bullet \cat C| \overset{|\Delta_r|}\ra |\Phi^{C_r} \sd_r N^\cyc_\bullet \cat C| \congar \rho_r^* \Phi^{C_r} |N^\cyc_\bullet \cat C| \]
When all the $\cat C(c_i,c_{i+1})$ are cofibrant, $\gamma_r$ is a realization of isomorphisms at each level, so $\gamma_r$ is an isomorphism.

Now we check compatibility.
The compatibility square may be expanded and subdivided
\[
\xymatrix @C=7em{
|N^\cyc_\bullet \cat C| \ar[r]^-{\Delta_{mn}} \ar[d]^-{\Delta_m} & \Phi^{C_{mn}} |\sd_{mn} N^\cyc_\bullet \cat C| \ar[d]^-{\textup{it}} \ar[r]^-{\Phi^{C_{mn}} D_{mn}}_-\cong & \Phi^{C_{mn}} |N^\cyc_\bullet \cat C| \ar[d]^-{\textup{it}} \\
\Phi^{C_m} |\sd_m N^\cyc_\bullet \cat C| \ar[d]^-{\Phi^{C_m} D_m}_-\cong & \Phi^{C_m} \Phi^{C_n} |\sd_{mn} N^\cyc_\bullet \cat C| \ar[d]^-{\Phi^{C_m} \Phi^{C_n} D_m}_-\cong \ar[r]^-{\Phi^{C_m} \Phi^{C_n} D_{mn}}_-\cong & \Phi^{C_m} \Phi^{C_n} |N^\cyc_\bullet \cat C| \ar@{=}[d] \\
\Phi^{C_m} |N^\cyc_\bullet \cat C| \ar[r]^-{\Phi^{C_m} \Delta_n} & \Phi^{C_m} \Phi^{C_n} |\sd_n N^\cyc_\bullet \cat C| \ar[r]^-{\Phi^{C_m} \Phi^{C_n} D_n}_-\cong & \Phi^{C_m} \Phi^{C_n} |N^\cyc_\bullet \cat C| \\}
\]
The top-right square commutes by naturality of the iterated fixed points map, and the bottom-right commutes by Prop \ref{cyclic_fixed_points_compatibility}.
The left-hand rectangle is subtle, so we expand and subdivide it once more:
\[ \xymatrix @C=5em{
|N^\cyc_\bullet \cat C| \ar[rr]^-{\Delta_{mn}} \ar[d]^-{\Delta_m}
 &
 & \Phi^{C_{mn}} |\sd_m \sd_n N^\cyc_\bullet \cat C| \ar[d]^-{\textup{it}} \\
\Phi^{C_m} |\sd_m N^\cyc_\bullet \cat C| \ar[d]^-{\Phi^{C_m} D_m}_-\cong \ar[r]^-{\Phi^{C_m} \sd_m \Delta_n}
 & \Phi^{C_m} |\sd_m \Phi^{C_n} \sd_n N^\cyc_\bullet \cat C| \ar[d]^-{\Phi^{C_m} D_m}_-\cong \ar[r]^-{\textup{int}}_-\cong
 & \Phi^{C_m} \Phi^{C_n} |\sd_m \sd_n N^\cyc_\bullet \cat C| \ar[d]^-{\Phi^{C_m} \Phi^{C_n} D_m}_-\cong \\
\Phi^{C_m} |N^\cyc_\bullet \cat C| \ar[r]^-{\Phi^{C_m} \Delta_n}
 & \Phi^{C_m} \Phi^{C_n} |\sd_n N^\cyc_\bullet \cat C| \ar@{=}[r]
 & \Phi^{C_m} \Phi^{C_n} |\sd_n N^\cyc_\bullet \cat C| \\
} \]
The bottom-left square inside commutes by naturality of $D_m$. The interchange map ``int'' is the obvious identification of the two cyclic spectra which at simplicial level $k-1$ are both given by $\Phi^{C_m} \Phi^{C_n} N^\cyc_{mnk-1} \cat C$. The lower-right square then easily commutes, and the remaining rectangle commutes by Prop \ref{normdiag_iterated_commute}.
\end{proof}


In order to do homotopy theory, we need to know which maps $\cat C \rightarrow \cat D$ are sent to weak equivalences $THH(\cat C) \rightarrow THH(\cat D)$, and we need conditions guaranteeing that $THH(\cat C)$ will be cofibrant.
By our work above, this reduces to a calculation of the latching maps and cyclic latching maps.
Let $\cat S$ denote the initial spectrally-enriched category on the objects of $\cat C$:
\[ \cat S(c_i,c_j) = \left\{ \begin{array}{rl} \Sph & c_i = c_j \\ * & c_i \neq c_j \end{array} \right. \]
The latching maps of the cyclic bar construction can be described concisely in terms of the canonical functor $\cat S \to \cat C$.

\begin{prop}\label{cyclic_bar_latching}
For every $n \geq 0$ the latching map $L_n N^\cyc \cat C \rightarrow N^\cyc_n \cat C$ is the wedge of pushout-products
\[ \bigvee_{c_0,\ldots,c_n \in \ob \cat C} (\cat S(c_0,c_1) \rightarrow \cat C(c_0,c_1)) \square \ldots \square (\cat S(c_{n-1},c_n) \rightarrow \cat C(c_{n-1},c_n)) \square (* \rightarrow \uline{\cat C(c_n,c_0)}) \]
and the cyclic latching map $L_n^\cyc N^\cyc_\bullet \cat C \rightarrow N^\cyc_n \cat C$ is the wedge of pushout-products
\[ \bigvee_{c_0,\ldots,c_n \in \ob \cat C} (\cat S(c_0,c_1) \rightarrow \cat C(c_0,c_1)) \square \ldots \square (\uline{\cat S(c_n,c_0)} \rightarrow \uline{\cat C(c_n,c_0)}). \]
\end{prop}
\begin{proof}
One proves by induction that the pushout-product of $n+1$ different maps $f_0: A_0 \rightarrow X_0, \ldots, f_n: A_n \rightarrow X_n$ comes from a cube-shaped diagram indexed by the subsets $S \subseteq \{0,\ldots,n\}$ and inclusions.
Each $S$ is assigned to the smash product of those $A_i$ for $i \not\in S$ and $X_i$ for $i \in S$.
The pushout-product $f_0 \square \ldots \square f_n$ is then the map that includes into the final vertex the colimit of the remaining vertices.

Therefore it suffices to identify the cube for the pushout-product with the cube from Prop \ref{cyclic_latching_cube} for the $n$th cyclic latching object $L_n^\cyc$. Each cube sends $S \subseteq \{0,\ldots,n\}$ to a smash product in which the smash summand for $(c_{i-1},c_i)$ is $\cat C(c_{i-1},c_i)$ if $i \in S$ and $\cat S(c_{i-1},c_i)$ if $i \not\in S$. In the pushout-product cube, the map induced by the inclusion $S \subseteq T$ is a smash product of $\cat S(c_{i-1},c_i) \to \cat C(c_{i-1},c_i)$ for each $i \in T - S$, together with the identity map on $\cat S(c_{i-1},c_i)$ for $i \not\in T$ and $\cat C(c_{i-1},c_i)$ for $i \in S$. But this is the same as the map in the cyclic latching cube, because the rounding down map $T \to S$ preserves every arrow which ends in $S$ and squashes the rest, so in the cyclic structure that induces a map that includes the unit for every arrow not ending in $S$ and preserves the rest.
Therefore the two cubes coincide.
Restricting attention to subsets $S$ containing 0 gives the cube for the simplicial latching object, giving a pushout-product in which the last factor is always $\uline{\cat C(c_n,c_0)}$.
\end{proof}

\begin{rmk}
We have claimed that the 0th cyclic latching map is the wedge of unit maps $\iota: \cat S(c,c) \to \cat C(c,c)$.
In general, this is not quite correct $-$ it is actually the wedge of inclusions of the images of these unit maps.
However the inclusion of the image of $\iota$ is still a pushout of $\iota$, so it does not matter which one we use in the latching square from Prop \ref{cyclic_latching_square}.
\end{rmk}

The previous proposition suggests that we need a very weak cofibrancy assumption on $\cat C$ to guarantee that $THH(\cat C)$ is well behaved.

\begin{df}
$\cat C$ is \emph{cofibrant} if every map $\cat S(c_i,c_j) \to \cat C(c_i,c_j)$ is a cofibration of orthogonal spectra. Equivalently, every $\cat C(c_i,c_j)$ is a cofibrant orthogonal spectrum.
\end{df}

\begin{prop}\label{cyclic_nerve_is_s1_cofibrant}
If $\cat C$ is cofibrant then $|N^\cyc_\bullet \cat C|$ is a cofibrant $S^1$-spectrum.
Moreover the inclusion of each cyclic skeleton into the next is a cofibration of $S^1$-spectra.
\end{prop}
\begin{proof}
By Prop \ref{reedy_cof_implies_realization_is_cof_s1_spectra}, it suffices to show that the cyclic latching map from Prop \ref{cyclic_bar_latching}
\[ \bigvee_{c_0,\ldots,c_{n-1} \in \ob \cat C} (\cat S(c_0,c_1) \ra \cat C(c_0,c_1)) \square \ldots \square (\uline{\cat S(c_{n-1},c_0)} \ra \uline{\cat C(c_{n-1},c_0)}) \]
is a $C_n$-cofibration of spectra.
We restrict to one wedge summand at a time and consider its $C_n$-orbit.
If there is no periodicity in the objects $c_0, \ldots, c_{n-1}$ then the orbit is of the form $(C_n)_+$ smashed with a pushout-product of cofibrations, so it is automatically a $C_n$-cofibration.
When there is $r$-fold periodicity, the problem instead reduces to showing that an $r$-fold pushout-product of a single cofibration $f$ of orthogonal spectra becomes a $C_r$-cofibration $f^{\square r}$.
Since $\square$ preserves retracts, it suffices to show that if $f$ is a cell complex of orthogonal spectra then $f^{\square r}$ is a cell complex of orthogonal $C_r$-spectra.

In fact, it is a cell complex of orthogonal $\Sigma_r$-spectra.
The argument for this is tedious but very formal.
It holds because the categories of orthogonal $G$-spectra with varying $G$ satisfy the following assumptions: the domains of our $G$-cells are small with respect to relative cell complexes; $\sma$ commutes with colimits in each variable; a pushout-product of an $H$-cell with a $K$-cell is a coproduct of $H \times K$-cells; the operation $G \sma_H -$ takes $H$-cells to $G$-cells; restriction of group actions takes $G$-cells to $H$-cell complexes; and the $n$-fold pushout-product of a single cell is a $\Sigma_n$-cell complex.
This last assumption can be observed for orthogonal spectra by combining the space-level argument (e.g. \cite[3.4]{malkiewich2015tower}) with the fact that an $n$-fold smash power of a free spectrum $F_{\R^m} A$ is isomorphic to $F_{\oplus^n \R^m} A^{\sma n}$ as a $\Sigma_n$-spectrum.
\end{proof}


\begin{prop}\label{prop:thh_equivariantly_derived}
If $\cat C$, $\cat D$ are cofibrant and $\cat C \to \cat D$ is a pointwise weak equivalence which is the identity on objects, then it induces an $\mc F$-equivalence of $S^1$-spectra $|N^\cyc_\bullet \cat C| \to |N^\cyc_\bullet \cat D|$ (see Defn \ref{df:f_equiv}).
\end{prop}
\begin{proof}
It is easy to check that $N^\cyc_\bullet \cat C \to N^\cyc_\bullet \cat D$ is a levelwise stable equivalence. By Prop \ref{cyclic_bar_latching}, both simplicial spectra are Reedy cofibrant, so the map of realizations is an equivalence of nonequivariant spectra. By Prop \ref{cyclic_nerve_is_s1_cofibrant}, both of these realizations are cofibrant $S^1$-spectra, and by Thm \ref{construction_of_THH} each one is naturally equivalent its own geometric fixed points. It follows that the map of left-derived geometric fixed points
\[ \Phi^{C_n}|N^\cyc_\bullet \cat C| \ra \Phi^{C_n}|N^\cyc_\bullet \cat D| \]
is an equivalence for all $n \geq 1$. By Prop \ref{prop:geometric_fixed_pts_measure_equivs} the map $|N^\cyc_\bullet \cat C| \to |N^\cyc_\bullet \cat D|$ is therefore an $\mathcal F$-equivalence.
\end{proof}

\section{Tensors and duals of cyclotomic spectra.}

In this final section, we discuss how to tensor and dualize cyclotomic structures, and use this to prove Theorem \ref{intro_dual_equivariant_thm}.

\subsection{A general framework for dualizing cyclotomic structures.}\label{cyclotomic}
Recall that a \emph{cyclotomic spectrum} is an orthogonal $S^1$-spectrum $T$ with compatible maps of $S^1$-spectra for all $n \geq 1$
\[ c_n: \rho_n^* \Phi^{C_n} T \ra T \]
for which the composite map
\begin{equation}\label{eq:derived_cn}
\rho_n^* \Phi^{C_n} (cT) \ra \rho^*_n \Phi^{C_n} T \ra T
\end{equation}
is an $\mc F$-equivalence of $S^1$-spectra (Defn \ref{df:f_equiv}).
Here $c$ refers to cofibrant replacement in the stable model structure on orthogonal $S^1$-spectra, Prop \ref{prop:model_structure}.
To be more specific about the compatibility, we require that for all $m,n \geq 1$ the square
\[ \xymatrix @C=7em{
\rho_{mn}^* \Phi^{C_{mn}} X \ar[r]^-{c_{mn}} \ar[d]^-{\textup{it}} & X \\
\rho_m^* \Phi^{C_m} \rho_n^* \Phi^{C_n} X \ar[r]^-{\rho_m^*\Phi^{C_m} c_n} & \rho_m^* \Phi^{C_m} X \ar[u]_-{c_m} } \]
commutes.
The left vertical is the canonical iterated fixed points map described in \cite{blumberg2013homotopy}, Prop 2.4, and it is an isomorphism when $X$ is cofibrant as an $S^1$-spectrum.

A \emph{pre-cyclotomic spectrum} has all the same structure except that the map \eqref{eq:derived_cn} need not be an equivalence.
An \emph{op-pre cyclotomic spectrum} has the above structure, but every map has the opposite direction, except for the iterated fixed points map.

In contrast to this, we give a more restrictive definition:
\begin{df}
A \emph{tight cyclotomic spectrum} is a cofibrant $S^1$-spectrum with isomorphisms $\gamma_n: T \congar \rho_n^* \Phi^{C_n} T$ of $S^1$-spectra for all $n \geq 0$ compatible in the following way:
\[ \xymatrix @C=7em{
T \ar[r]^-{\gamma_{mn}}_-\cong \ar[d]^-{\gamma_m}_-\cong & \rho_{mn}^* \Phi^{C_{mn}} T \ar[d]^-{\textup{it}}_-\cong \\
\rho_m^* \Phi^{C_m} T \ar[r]^-{\rho_m^*\Phi^{C_m} \gamma_n}_-\cong & \rho_m^* \Phi^{C_m} \rho_n^* \Phi^{C_n} T } \]
\end{df}
Here ``cofibrant'' means in the stable model structure of Prop \ref{prop:model_structure}.
This implies that the geometric fixed points compute the left-derived geometric fixed points, i.e. the first map of \eqref{eq:derived_cn} is always an equivalence.
So a tight cyclotomic spectrum may be regarded as a cyclotomic spectrum by taking $c_n = \gamma_n^{-1}$ and forgetting that it is an isomorphism.
We can summarize most of the previous section in a single theorem:
\begin{thm}\label{recall_cyclic_nerves_are_tight}
If $R$ is an orthogonal ring spectrum, which is cofibrant as an orthogonal spectrum, then $THH(R)$ is a tight cyclotomic spectrum.
If $\cat C$ is a cofibrant spectral category, then $THH(\cat C)$ is a tight cyclotomic spectrum.
\end{thm}

The point of these definitions is to dualize cyclotomic structures.
Our first result is
\begin{prop}
If $T$ is a tight cyclotomic spectrum and $T'$ is pre-cyclotomic then the function spectrum $F(T,T')$ has a natural pre-cyclotomic structure.
\end{prop}
\begin{cor}
If $T$ is a tight cyclotomic spectrum then the functional dual $DT = F(T,\Sph)$ is pre-cyclotomic.
\end{cor}
\begin{proof}
We define the structure map $c_r$ as the composite
\[ \rho_r^* \Phi^{C_r} F(T,T') \overset{\overline\alpha}\ra F(\rho_r^* \Phi^{C_r} T, \rho_r^* \Phi^{C_r} T') \overset{F(\gamma_r,c_r)} \ra F(T,T') \]
where $\overline\alpha$ is the ``restriction'' map adjoint to
\[ \rho_r^* \Phi^{C_r} F(T,T') \sma \rho_r^* \Phi^{C_r} T \overset{\alpha}\ra \rho_r^* \Phi^{C_r} (F(T,T') \sma T) \ra \rho_r^* \Phi^{C_r} T' \]
and $\alpha$ is the usual commutation of $\Phi^{C_r}$ with smash products.
By the usual rules for equivariant adjunctions, $c_r$ is automatically $S^1$-equivariant.
We verify that these maps are compatible.
Clearly they are natural in $T$ and $T'$, so in the diagram
\[ \resizebox{\textwidth}{!}{\xymatrix{
\rho_{mn}^* \Phi^{C_{mn}} F(T,T') \ar[r]^-{\overline\alpha} \ar[d]^-{\textup{it}}
 & F(\rho_{mn}^* \Phi^{C_{mn}} T,\rho_{mn}^* \Phi^{C_{mn}} T') \ar[r]^-{F(\id,\textup{it})}
 & F(\rho_{mn}^* \Phi^{C_{mn}} T,\rho_m^* \Phi^{C_m} \rho_n^* \Phi^{C_n} T') \ar@{<-}[d]_-\cong^-{F(\textup{it},\id)} \\
\rho_m^* \Phi^{C_m} \rho_n^* \Phi^{C_n} F(T,T') \ar[r]^-{\Phi^{C_m}\overline\alpha}
 & \rho_m^* \Phi^{C_m} F(\rho_n^* \Phi^{C_n} T,\rho_n^* \Phi^{C_n} T') \ar[r]^-{\overline\alpha} \ar[d]^-{F(\Phi^{C_m} \gamma_n,\Phi^{C_m} c_n)}
 & F(\rho_m^* \Phi^{C_m} \rho_n^* \Phi^{C_n} T,\rho_m^* \Phi^{C_m} \rho_n^* \Phi^{C_n} T') \ar[d]^-{F(\Phi^{C_m} \gamma_n,\Phi^{C_m} c_n)} \\
 & \rho_m^* \Phi^{C_m} F(T,T') \ar[r]^-{\overline\alpha}
 & F(\rho_m^* \Phi^{C_m} T,\rho_m^* \Phi^{C_m} T') \ar[d]^-{F(\gamma_m,c_m)} \\
 && F(T,T') }} \]
the small square automatically commutes.
The left-most and right-most paths compose to give the two maps we are trying to compare.
So, we just need to show that the big rectangle at the top commutes.
It is adjoint to
\[ \resizebox{\textwidth}{!}{\xymatrix{
\rho_{mn}^* \Phi^{C_{mn}} F(T,T') \sma \rho_{mn}^* \Phi^{C_{mn}} T \ar[d]^-{\textup{it} \sma \textup{it}} \ar[r]^-{\alpha}
 & \rho_{mn}^* \Phi^{C_{mn}} (F(T,T') \sma T) \ar[r] \ar[d]^-{\textup{it}}
 & \rho_{mn}^* \Phi^{C_{mn}} T' \ar[d]^-{\textup{it}} \\
\rho_m^* \Phi^{C_m} \rho_n^* \Phi^{C_n} F(T,T') \sma \rho_m^* \Phi^{C_m} \rho_n^* \Phi^{C_n} T \ar[r]^-{\alpha \circ \alpha}
 & \rho_m^* \Phi^{C_m} \rho_n^* \Phi^{C_n} (F(T,T') \sma T) \ar[r]
 & \rho_m^* \Phi^{C_m} \rho_n^* \Phi^{C_n} T'
}} \]
The right square is by naturality of the iterated fixed points map, and the left square is by Prop \ref{smash_iterated_commute}.
\end{proof}

We have chosen to state these results for tight cyclotomic spectra, because then every object we work with has pre-cyclotomic structure, as opposed to a mix of objects with pre-cyclotomic and op-pre-cyclotomic structure. If we freely allow ourselves to use both structures then the we get the following more general conclusion. It tells us that we have something close to, but not quite, a closed symmetric monoidal category of spectra with these structures.

\begin{prop}
If $X$ and $Y$ are op-pre cyclotomic spectra and $Z$ is a pre-cyclotomic spectrum then $X \sma Y$ is op-pre cyclotomic, $F(Y,Z)$ is pre-cyclotomic, and the adjunction
\[ F(X \sma Y,Z) \cong F(X,F(Y,Z)) \]
respects the pre-cyclotomic structure.
\end{prop}

\begin{proof}
The above proof generalizes to show that $F(Y,Z)$ is pre-cyclotomic, since we only used the maps $\gamma_n$ for $Y$ and $c_n$ for $Z$. For $X \sma Y$ we define the op-cyclotomic structure by
\[ \xymatrix{ X \sma Y \ar[r]^-{\gamma_n \sma \gamma_n} & \Phi^{C_n} X \sma \Phi^{C_n} Y \ar[r]^-\alpha & \Phi^{C_n} (X \sma Y) } \]
where the $\rho_n^*$s are suppressed.
By an easy diagram chase, the compatibility reduces again to Prop \ref{smash_iterated_commute}.
When we check that the adjunction preserves the cyclotomic structures, we reduce to the claim that the interchange map $\alpha$ has an associativity property. This can be proven from the definitions with a little bit of work, but it also follows effortlessly from the rigidity theorem.
\end{proof}

This analysis does not quite apply to the categories of pre-cyclotomic or cyclotomic spectra, because we get zig-zags when we try to define a cyclotomic structure on their tensor product. However this problem goes away if we restrict attention to cofibrant objects, so we can draw a conclusion about the homotopy category:

\begin{prop}\label{tensor_triangulated}
The homotopy categories of pre-cyclotomic spectra and of cyclotomic spectra from \cite{blumberg2013homotopy} have a tensor triangulated structure.
\end{prop}

\begin{proof}
For simplicity we suppress $\rho_n^*$.
If $X$ and $Y$ are cofibrant (pre-)cyclotomic spectra, we make $X \sma Y$ into a (pre-)cyclotomic spectrum using the structure maps
\[ \xymatrix{ \Phi^{C_n}(X \sma Y) \ar@{<-}[r]^-\alpha_-\cong & \Phi^{C_n} X \sma \Phi^{C_n} Y \ar[r]^-{c_n \sma c_n} & X \sma Y } \] 
The relevant compatibility square is
\[ \xymatrix @C=6em{
\Phi^{C_{mn}}(X \sma Y) \ar[dd]^-{\textup{it}} \ar@{<-}[r]^-\alpha_-\cong &
 \Phi^{C_{mn}} X \sma \Phi^{C_{mn}} Y \ar[r]^-{c_{mn} \sma c_{mn}} \ar[d]^-{\textup{it} \sma \textup{it}} &
 X \sma Y \\
&
 \Phi^{C_m} \Phi^{C_n} X \sma \Phi^{C_m} \Phi^{C_n} Y \ar[r]^-{\Phi^{C_m} c_n \sma \Phi^{C_m} c_n} \ar[d]^-\alpha &
 \Phi^{C_m} X \sma \Phi^{C_m} Y \ar[u]^-{c_m \sma c_m} \\
\Phi^{C_m} \Phi^{C_n} (X \sma Y) \ar@{<-}[r]^-{\Phi^{C_m} \alpha}_-\cong &
 \Phi^{C_m} (\Phi^{C_n} X \sma \Phi^{C_n} Y) \ar[r]^-{\Phi^{C_m}(c_n \sma c_n)} &
 \Phi^{C_m} (X \sma Y) \ar@{<-}[u]^-\alpha_-\cong
} \]
Again Prop \ref{smash_iterated_commute} gives us the left-hand rectangle, the top-right square is the smash product of the compatibility squares for $X$ and $Y$, and the bottom-right commutes by naturality of $\alpha$. It is straightforward to check that this smash product preserves colimits and cofibers of (pre-)cyclotomic spectra, so this gives the desired tensor triangulated structure on the homotopy category.
\end{proof}

\begin{rmk}
The analogue of this theorem for $p$-pre-cyclotomic spectra and $p$-cyclotomic spectra is also true, and it is much easier.
\end{rmk}

Returning to the pre-cyclotomic structure on $F(T,T')$, our main example of interest will be when $T = |N^\cyc \cat C|$ is the cyclic nerve of a ring or category.
We have just proven that $F(|N^\cyc \cat C|,T')$ is a pre-cyclotomic spectrum. 
It is also the totalization of the cocyclic $S^1$-spectrum
\[ Y^k = F(N^\cyc_k \cat C,T') \]
To be precise, the $S^1$ is acting only on the $T'$, and $\mathbf\Lambda$ is acting by the dual of the $\mathbf\Lambda^\op$ action on $N^\cyc_k \cat C$.
This puts two commuting $S^1$-actions on the totalization, but we restrict attention to the diagonal $S^1$-action, because this is the action that agrees with the pre-cyclotomic structure we just defined.

In order to compare this to the cocyclic spectrum $\Sigma^\infty_+ X^{\bullet + 1}$, we will need to describe our cyclotomic structure maps using only the cocyclic structure on $F(|N^\cyc \cat C|,T')$:
\begin{prop}\label{cyclotomic_str_in_terms_of_tot}
The cyclotomic structure map on $\Tot(Y^\bullet) \cong F(|N^\cyc \cat C|,T')$ is equal to the composite of $S^1$-equivariant maps
\begin{eqnarray*}
\rho_r^* \Phi^{C_r} \Tot(F(N^\cyc_\bullet \cat C,T')) 
 &\overset{D_r}\ra & \rho_r^* \Phi^{C_r} \Tot(F(\sd_r N^\cyc_\bullet \cat C,T')) \\
 &\ra & \rho_r^* \Tot(\Phi^{C_r} F(\sd_r N^\cyc_\bullet \cat C,T')) \\
 &\overset{\overline\alpha}\ra & \rho_r^* \Tot(F(P_r \Phi^{C_r} \sd_r N^\cyc_\bullet \cat C,\Phi^{C_r} T')) \\
 &\congar & \Tot(F(\Phi^{C_r} \sd_r N^\cyc_\bullet \cat C,\rho_r^* \Phi^{C_r} T')) \\
 &\overset{F(\Delta,c_r)}\ra & \Tot(F(N^\cyc_\bullet \cat C,T')) \\
\end{eqnarray*}
where the undecorated map is the interchange of Prop \ref{tot_phi_interchange}.
\end{prop}
\begin{proof}
We compare to the structure map we defined above:
\[ \resizebox{\textwidth}{!}{\xymatrix @C=2em{
\rho_r^* \Phi^{C_r} F(|N^\cyc_\bullet \cat C|,T') \ar[rr]^-\cong \ar[d]^-{\overline\alpha} \ar[rd]^-{D_r}_-\cong
 &
 & \rho_r^* \Phi^{C_r} \Tot(F(N^\cyc_\bullet \cat C,T')) \ar[d]^-{D_r} \\
 F(\rho_r^* \Phi^{C_r} |N^\cyc_\bullet \cat C|,\rho_r^* \Phi^{C_r} T') \ar[rd]^-{D_r}_-\cong
 & \rho_r^* \Phi^{C_r} F(|\sd_r N^\cyc_\bullet \cat C|,T') \ar[d]^-{\overline\alpha} \ar[r]^-\cong
 & \rho_r^* \Phi^{C_r} \Tot(F(\sd_r N^\cyc_\bullet \cat C,T')) \ar[d] \\
 & F(\rho_r^* \Phi^{C_r} |\sd_r N^\cyc_\bullet \cat C|,\rho_r^* \Phi^{C_r} T') \ar[d]^-\cong
 & \rho_r^* \Tot(\Phi^{C_r} F(\sd_r N^\cyc_\bullet \cat C,T')) \ar[d]^-{\overline\alpha} \\
 & F(\rho_r^* |P_r \Phi^{C_r} \sd_r N^\cyc_\bullet \cat C|,\rho_r^* \Phi^{C_r} T') \ar[d]^-\cong \ar[r]^-\cong
 & \rho_r^* \Tot(F(P_r\Phi^{C_r} \sd_r N^\cyc_\bullet \cat C,\Phi^{C_r} T')) \ar[d]^-\cong \\
 & F(|\Phi^{C_r} \sd_r N^\cyc_\bullet \cat C|,\rho_r^* \Phi^{C_r} T') \ar[d]^-{F(\Delta,c_r)} \ar[r]^-\cong
 & \Tot(F(\Phi^{C_r} \sd_r N^\cyc_\bullet \cat C,\rho_r^* \Phi^{C_r} T')) \ar[d]^-{F(\Delta,c_r)} \\
 & F(|N^\cyc_\bullet \cat C|,T') \ar[r]^-\cong
 & \Tot(F(N^\cyc_\bullet \cat C,\Phi^{C_r} T'))
}} \]
Most of these squares commute easily.
The nontrivial one in the middle can be simplified to the following: if $X_\bullet$ is a simplicial $C_r$-spectrum and $T$ is a $C_r$-spectrum then
the middle rectangle of
\[ \xymatrix @R=1em{
F_{W^{C_r}} S^0 \sma \Map_*^{C_r}(|X_\bullet|, \sh^W T) \ar[d] \\
 \Phi^{C_r} F(|X_\bullet|, T) \ar[r]^-\cong \ar[d]^-{\overline\alpha} & \Phi^{C_r} \Tot(F(X_\bullet,T)) \ar[d] \\
 F(\Phi^{C_r} |X_\bullet|, \Phi^{C_r} T) \ar[d]^-\cong & \Tot(\Phi^{C_r} F(X_\bullet,T)) \ar[d]^-{\overline\alpha} \\
 F(|\Phi^{C_r} X_\bullet|, \Phi^{C_r} T) \ar[r]^-\cong & \Tot(F(\Phi^{C_r} X_\bullet,\Phi^{C_r} T)) \ar[d] \\
 & F(F_{V^{C_r}} S^0 \sma \Delta^k_+ \sma X_k(V)^{C_r}, \Phi^{C_r} T)
} \]
commutes. Here $\sh^W T$ is shorthand for the mapping spectrum $F(F_W S^0,T)$, which is used in the following standard formula for level $W$ of a mapping spectrum:
\[ F(|X_\bullet|, T)(W) \cong F(|X_\bullet|, \sh^W T)(0) \cong \Map_*(|X_\bullet|, \sh^W T) \]
Now, it suffices to show that the composites from the top to the bottom of our rectangle are identical, for any $C_r$-representations $V$ and $W$ and any integer $k \geq 0$.
For the left-hand branch it is easy to check this is adjoint to the composite
\[ \xymatrix @R=1em{
F_{V^{C_r}} S^0 \sma \Delta^k_+ \sma X_k(V)^{C_r} \sma F_{W^{C_r}} S^0 \sma \Map_*^{C_r}(|X_\bullet|, \sh^W T) \ar[d]^-{\textup{include into }\Phi^{C_r}} \\
\Phi^{C_r} (\Delta^k_+ \sma X_k) \sma \Phi^{C_r}F(|X_\bullet|,T) \ar[d]^-{\textup{include into }|X_\bullet|} \\
\Phi^{C_r} |X_\bullet| \sma \Phi^{C_r}F(|X_\bullet|,T) \ar[d]^-\alpha \\
\Phi^{C_r} (|X_\bullet| \sma F(|X_\bullet|,T)) \ar[d]^-{\Phi^{C_r}(\ev)} \\
\Phi^{C_r} T
} \]
A careful trace through the diagram in Prop \ref{tot_phi_interchange} shows that the right-hand branch is the composite
\[ \begin{array}{rcl}
F_{W^{C_r}} S^0 \sma \Map_*^{C_r}(|X_\bullet|, \sh^W T)
&\overset{\textup{restrict}}\ra & F_{W^{C_r}} S^0 \sma \Map_*^{C_r}(\Delta^k_+ \sma X_k, \sh^W T) \\
&\overset{\textup{assembly}}\ra & F(\Delta^k_+, F_{W^{C_r}} S^0 \sma \Map_*^{C_r}(X_k, \sh^W T)) \\
&\overset{\textup{include}}\ra & F(\Delta^k_+, \Phi^{C_r} F(X_k, T)) \\
&\overset{\overline\alpha}\ra & F(\Delta^k_+, F(\Phi^{C_r} X_k, \Phi^{C_r} T)) \\
&\overset{\textup{include}}\ra & F(\Delta^k_+, F(F_{V^{C_r}} S^0 \sma X_k(V)^{C_r}, \Phi^{C_r} T))
\end{array} \]
The adjoint of this map does indeed agree with the first, by a very long diagram-chase.
The essential ingredients are functoriality of $\Phi^{C_r}$, naturality of $\alpha$ and $\ev$, and associativity of $\alpha$.
\end{proof}

Now we know that $F(T,T')$ has a pre-cyclotomic structure.
This won't be very useful unless we can make cofibrant and fibrant replacements of $T$ and $T'$, respectively, while preserving that structure.
For this task, we use the model structure on cyclotomic and pre-cyclotomic spectra defined in \cite{blumberg2013homotopy}.
It has following attractive property.
\begin{lem}
If $T$ is cofibrant or fibrant in the model* category on (pre)cyclotomic spectra, then it is also cofibrant or fibrant, respectively, as an orthogonal $S^1$-spectrum in the $\mc F$-model structure.
\end{lem}
\begin{proof}
The fibrant part is true by definition.
For the cofibrant part it suffices to check that the monad
\[ \C X = \bigvee_{n \geq 1} \rho_n^* \Phi^{C_n} X \]
preserves cofibrant objects in the $\mc F$-model structure.
This is true because wedge sums, geometric fixed points, and change of groups all preserve cofibrations.
\end{proof}

In light of this fact, we can replace $T'$ with a fibrant cyclotomic spectrum $fT'$, resulting in the pre-cyclotomic spectrum $F(T,fT')$ whose underlying $S^1$-spectrum has the homotopy type of the derived mapping spectrum from $T$ to $T'$ (i.e. the first input is cofibrant and the second input is fibrant).
Specializing to $T' = \Sph$ gives a pre-cyclotomic structure on the dual $F(T,f\Sph)$.

\begin{rmk}
If $T$ is finite as a genuine $S^1$ spectrum, then $F(T,f\Sph)$ is actually cyclotomic, not just pre-cyclotomic.
In general, however, this is not true.
One can check that $T = \Sigma^\infty_+ \mathbb{RP}^\infty$ gives a counterexample.
In the next section we will consider an example where $T$ is infinite, but $F(T,f\Sph)$ is still cyclotomic, mainly for reasons of connectivity.
\end{rmk}

\subsection{The equivariant duality between $THH(DX)$ and $\Sigma^\infty_+ LX$.}

Let $X$ be a finite based CW complex and let $DX = F(X_+,\Sph)$ denote its Spanier-Whitehead dual.
Though $\Sph$ is not fibrant, $X$ is compact, so $DX$ has the correct homotopy type.
It is also finite, of course, but it is no longer compact, and this slightly complicates our proof below.

$DX$ is a commutative ring with multiplication given by the dual of the diagonal map of $X$. Likewise, the spectrum $\ti DX = F(X,\Sph)$ has a commutative multiplication given by the dual of the smash diagonal $X \to X \sma X$. It does not have a unit, but we can make $\Sph \vee \ti DX$ into a ring spectrum by having $\Sph$ act as the unit. The levelwise fiber sequence of spectra
\[ F(X,\Sph) \ra F(X_+,\Sph) \ra \Sph \]
preserves the multiplications, and this allows us to form an equivalence of ring spectra
\[ \Sph \vee \ti DX \simar DX. \]
Let $c\ti DX$ denote cofibrant replacement of $DX$ as a unitless ring, so that
\[ c DX := \Sph \vee c\ti DX \ra \Sph \vee \ti DX \]
is a particularly nice cofibrant replacement of ring spectra.
We'll take as our example of a tight cyclotomic spectrum
\[ T = THH(cDX) \]
We recall that this cofibrant replacement ensures that $THH(cD(-))$ is homotopy invariant (Proposition \ref{prop:thh_equivariantly_derived}). Our starting point is the following consequence of Prop \ref{LX_is_derived_tot}. In its statement, we assume implicitly that cofibrant replacements are taken before each application of $D$ or $THH$.

\begin{thm}[\cite{kuhn2004mccord},\cite{campbell2014derived}]\label{nonequivariant_duality_thm}
When $X$ is a finite simply-connected CW complex, there is an equivalence of spectra with an $S^1$-action
\[ D(THH(DX)) \simeq THH(\Sigma^\infty_+ \Omega X) \simeq \Sigma^\infty_+ LX \]
in which $LX = \Map(S^1,X)$ is the free loop space.
\end{thm}

\begin{rmk}
If $M$ is a manifold then $DM \simeq M^{-TM}$ is a Thom spectrum.
But the analysis of \cite{blumberg2008topological} does not apply, because the multiplication on $M^{-TM}$ does not arise from the normal bundle $M \to BO$ being a loop map.
\end{rmk}

We will spend the rest of this section proving a more highly-structured version of that result:
\begin{thm}
Let $f\Sph$ be a fibrant replacement of $\Sph$ as a cyclotomic spectrum.
Then for every unbased space $X$ there is a natural map of pre-cyclotomic spectra
\[ \Sigma^\infty_+ LX \ra F(THH(cDX),f\Sph) \]
The left-hand side is always cyclotomic.
When $X$ is a finite simply-connected CW complex, the right-hand side is cyclotomic and the map is an $\mc F$-equivalence.
\end{thm}
\begin{cor}
When $X$ is a finite simply-connected CW complex, the equivalence between $THH(\Sigma^\infty_+ \Omega X)$ and the functional dual of $THH(DX)$ is an equivalence of cyclotomic spectra.
\end{cor}

\begin{proof}
We will describe explicitly the map of Thm \ref{nonequivariant_duality_thm} and check that it respects the pre-cyclotomic structures.
Then we will use connectivity arguments to argue that these pre-cyclotomic spectra are actually cyclotomic when $X$ is finite.

As above, let $Y^\bullet$ denote the cocyclic $S^1$-spectrum
\[ Y^k = F(N^\cyc_k cDX,f\Sph) = F((cDX)^{\sma (k+1)},f\Sph) \]
The totalization of $Y^\bullet$ is isomorphic to $F(|N^\cyc cDX|, f\Sph)$, and Prop \ref{cyclotomic_str_in_terms_of_tot} gives us a recipe for the pre-cyclotomic structure. Furthermore $Y^\bullet$ is the dual of a Reedy cofibrant simplicial spectrum, and is therefore Reedy fibrant.

We will construct a map $\Sigma^\infty_+ LX \to \Tot(Y^\bullet)$ by going through an intermediary $\Tot(Z^\bullet)$.
Let $Z^\bullet$ be the cocyclic spectrum $\Sigma^\infty_+ \Map(S^1_\bullet,X)$, so that
\[ Z^k = \Sigma^\infty_+ \Map(\mathbf\Lambda([k],[0]),X) \cong \Sigma^\infty_+ X^{k+1} \]
with $\mathbf\Lambda$ action given by applying $\Sigma^\infty_+$ to the usual $\mathbf\Lambda^\op$ action on the $\mathbf\Lambda(-,[0])$ term.
The interchange of Prop \ref{LX_is_derived_tot} gives a map of spectra
\[ \Sigma^\infty_+ LX \ra \Tot(Z^\bullet) \]
Next we construct a map of cocyclic spectra $Z^\bullet \to Y^\bullet$.
The evaluation map composed with the product in $\Sph$ and fibrant replacement
\[ \left(\Sigma^\infty_+ X\right)^{\sma (k+1)} \sma c(DX)^{\sma (k+1)} \ra\left(\Sigma^\infty_+ X\right)^{\sma (k+1)} \sma (DX)^{\sma (k+1)} \]
\[ \ra (\Sph)^{\sma (k+1)} \ra \Sph \ra f\Sph \]
is adjoint to a map
\[ Z^k = \Sigma^\infty_+ X^{k+1} \ra F((cDX)^{\sma (k+1)},f\Sph) = Y^k \]
Of course, this map is actually an equivalence when $X$ is finite. The map clearly commutes with the $S^1$-action on each level coming from $f\Sph$.
We check that it commutes with the cocyclic structure: for each $\gamma \in \mathbf\Lambda([k],[\ell])$ we have the square
\[ \xymatrix{
\Map(\Lambda[0]_k,X) \cong X^{k+1} \ar[r] \ar[d]^-\gamma & F((cDX)^{\sma k+1},f\Sph) \ar[d]^-\gamma \\
\Map(\Lambda[0]_\ell,X) \cong X^{\ell+1} \ar[r] & F((cDX)^{\sma \ell+1},f\Sph) } \]
which commutes if this one commutes:
\[ \xymatrix{
X^{k+1} \sma (cDX)^{\ell+1} \ar[r]^-{\gamma \sma \id} \ar[d]^-{\id \sma \gamma} & X^{\ell+1} \sma (cDX)^{\ell+1} \ar[d] \\
X^{k+1} \sma (cDX)^{k+1} \ar[r] & \Sph } \]
Both branches have the same description: $\gamma$ gives a map from a necklace with $k+1$ beads and every segment labeled by $X$ to a necklace with $\ell+1$ beads and every segment labeled by $DX$.
Each copy of $X$ is sent by $\gamma$ to a string of $a$ copies of $DX$; we apply the diagonal to $X_+ \overset\Delta\to (\prod^a X)_+$ and pair with those $a$ copies of $DX$.

Therefore we have a map of cocyclic $S^1$-spectra $Z^\bullet \to Y^\bullet$, with $S^1$ acting trivially on each cosimplicial level of $Z^\bullet$. Composing with the interchange map of Proposition \ref{LX_is_derived_tot} gives an $S^1$-equivariant map
\begin{equation}\label{lx_to_ybullet}
\Sigma^\infty_+ LX \ra \Tot(Z^\bullet) \ra \Tot(Y^\bullet).
\end{equation}
When $X$ is finite, this is the equivalence of Theorem \ref{nonequivariant_duality_thm}. In fact, when $X$ is finite the map $Z^\bullet \to Y^\bullet$ is an equivalence on each cosimplicial level, and we may therefore consider $Y^\bullet$ to be a Reedy fibrant replacement of $Z^\bullet$, so \eqref{lx_to_ybullet} is also a model for the derived interchange map of Proposition \ref{LX_is_derived_tot}.


Our next task is to check that the map \eqref{lx_to_ybullet} respects the pre-cyclotomic structures on the two ends.
The recipe in Prop \ref{cyclotomic_str_in_terms_of_tot} actually defines a pre-cyclotomic structure on $\Tot(Z^\bullet)$ as well, so our problem breaks up into two steps:
\begin{equation}\label{big_cyclotomic_structure_matchup}
\xymatrix{
 \Phi^{C_r} \Sigma^\infty_+ LX \ar[ddd]^-\cong \ar[r]^-\cong & \rho_r^* \Phi^{C_r} \Tot(Z^\bullet) \ar[r] \ar[d]^-{D_r}_-\cong & \rho_r^* \Phi^{C_r} \Tot(Y^\bullet) \ar[d]^-{D_r}_-\cong \\
 & \rho_r^* \Phi^{C_r} \Tot(\sd_r Z^\bullet) \ar[r] \ar[d] & \rho_r^* \Phi^{C_r} \Tot(\sd_r Y^\bullet) \ar[d] \\
 & \rho_r^*\Tot(\Phi^{C_r} \sd_r Z^\bullet) \ar[r] \ar@{<-}[d]^-\Delta_-\cong & \rho_r^* \Tot(\Phi^{C_r} \sd_r Y^\bullet) \ar[d]^-{F(\Delta,c_r) \circ \overline\alpha} \\
 \Sigma^\infty_+ LX \ar[r]^-\cong & \Tot(Z^\bullet) \ar[r] & \Tot(Y^\bullet)
}
\end{equation}
We start with the left-hand rectangle of (\ref{big_cyclotomic_structure_matchup}), where everything is a suspension spectrum and so all maps are completely determined by what they do at spectrum level 0.
The horizontal homeomorphisms may be computed by observing that $\Lambda[0]_k = \mathbf\Lambda([k],[0])$ has $(k+1)$ points $f_0,\ldots,f_k$, where $f_i: \Z \to \Z$ sends $0$ through $i-1$ to $0$ and $i$ through $k$ to $1$ (or if $i = 0$ it sends $0$ through $k$ to $0$).
Using our choice of homeomorphism $|\Lambda[0]| \cong \R/\Z$ from section \ref{sec:cyclic}, the $k$-simplex given by $f_i$ maps down to the circle $\R/\Z$ by the formula
\[ (t_0,\ldots,t_k) \mapsto (t_i + \ldots + t_k) \sim (1 - (t_0 + \ldots + t_{i-1})) \]
Negating the circle and reparametrizing $\Delta^k \subset \R^k$ as points $(x_1,\ldots,x_k)$ for which $0 \leq x_1 \leq x_2 \leq \ldots \leq x_k \leq 1$ according to the rule $x_i = t_0 + \ldots + t_{i-1}$, we arrive at the simple rule
\[ (f_i,x_1,\ldots,x_k) \mapsto x_i, \qquad x_0 := 0 \]
So now the map $LX \to \Tot(X^{\bullet + 1})$ can be expressed by the formula
\[ \Delta^{k-1} \times LX \ra X^{k} \]
\[ (r_1,\ldots,r_{k-1},\gamma) \mapsto (\gamma(0),\gamma(r_1),\ldots,\gamma(r_{k-1})) \]
as in \cite{cohen2002homotopy}.
Under this change of coordinates, both branches give
\begin{eqnarray*}
\gamma(-) &\ra & (r_1,\ldots,r_{k-1}) \mapsto (\gamma(0),\gamma(\frac1r r_1), \gamma(\frac1r r_2), \ldots, \gamma(\frac1r r_{k-1}),\gamma(0),\gamma(\frac1r r_1), \ldots)
\end{eqnarray*}
and so the square commutes.

Returning to (\ref{big_cyclotomic_structure_matchup}), the top and middle squares of the right-hand row automatically commute by the naturality of the cosimplicial diagonal and the interchange map with geometric fixed points.
The final square is then
\[ \xymatrix{
 \Tot(\Phi^{C_r} \sd_r Z^\bullet) \ar[r] \ar@{<-}[d]^-\Delta_-\cong & \rho_r^* \Tot(\Phi^{C_r} \sd_r Y^\bullet) \ar[d]^-{F(\Delta,c_r) \circ \overline\alpha} \\
 \Tot(Z^\bullet) \ar[r] & \Tot(Y^\bullet)
} \]
The map $\Delta$ is the cocyclic map
\[ \Phi^{C_r} \Sigma^\infty_+ X^{rk} \lcongar \Sigma^\infty_+ X^{k} \]
given by the Hill-Hopkins-Ravenel diagonal; this is almost tautologically cosimplicial.
The map $F(\Delta,c_r) \circ \overline\alpha$ is also cocyclic, so to check that this square commutes it suffices to check level $k-1$.
This boils down to this rectangle:
\[ \xymatrix{
\Phi^{C_r} \Sigma^\infty_+ X^{rk} \sma \Phi^{C_r} (cDX)^{\sma rk} \ar[r]^-\alpha & \Phi^{C_r} (\Sigma^\infty_+ X^{rk} \sma (cDX)^{\sma rk}) \ar[r] & \Phi^{C_r} \Sph \ar[d]^-\cong \\
\Sigma^\infty_+ X^{k} \sma (cDX)^{\sma k} \ar[u]_-{\Delta \sma \Delta} \ar[ur]_-\Delta \ar[rr] && \Sph } \]
The top triangle commutes because the norm diagonal commutes with smash products.
The trapezoid commutes because the inverse of the right-hand isomorphism is the norm diagonal on $\Sph$ (in fact there is only one isomorphism $\Sph \to \Sph$), and the norm diagonal is natural.
This finishes the proof that $\Sigma^\infty_+ LX \to \Tot(Y^\bullet)$ is a map of pre-cyclotomic spectra.

For the second phase of the proof, we assume that $X$ is finite and 1-connected, and we check that $\Tot(Y^\bullet)$ is actually cyclotomic; in other words the map
\[ \Phi^{C_r} \Tot(Y^\bullet) \ra \Tot(Y^\bullet) \]
is nonequivariantly an equivalence when $\Phi^{C_r}$ is left-derived.
For simplicity, we may forget the $S^1$-actions and remember only the cosimplicial $C_r$-action on $\sd_r Y^\bullet$, making it a cosimplicial $C_r$-spectrum.
Then our structure maps respect the restriction to the $k$-skeleton for each $k \geq 0$:
\begin{equation}\label{dual_is_cyclotomic_one_level_at_a_time}
\resizebox{\textwidth}{!}{\xymatrix @C=2em{
\Phi^{C_r} cF(|\sd_r N^\cyc_\bullet cDX|,f\Sph) \ar[r]^-{\overline\alpha} \ar[d] & F(\Phi^{C_r}|\sd_r N^\cyc_\bullet cDX|,\Phi^{C_r}f\Sph) \ar[r]^-{F(\Delta,c_r)} \ar[d] & F(|N^\cyc_\bullet cDX|,f\Sph) \ar[d] \\
\Phi^{C_r} cF(|\sk_k \sd_r N^\cyc_\bullet cDX|,f\Sph) \ar[r]^-{\overline\alpha} & F(\Phi^{C_r}|\sk_k \sd_r N^\cyc_\bullet cDX|,\Phi^{C_r}f\Sph) \ar[r]^-{F(\Delta,c_r)} & F(|\sk_k N^\cyc_\bullet cDX|,f\Sph)  }}
\end{equation}
We first argue that the bottom horizontal composite is an equivalence for each value of $k$. The skeleta $|\sk_k N^\cyc_\bullet cDX|$ and $|\sk_k \sd_r N^\cyc_\bullet cDX|$ all have the homotopy type of a finite spectrum, so by \cite[III.1.9]{lewis1986equivariant}, the interchange map $\overline\alpha$ is an equivalence. Of course, the diagonal isomorphism $\Delta$ from the proof of Theorem \ref{construction_of_THH} is an isomorphism of simplicial objects, so it also gives an isomorphism of skeleta. This is enough to conclude that the map $F(\Delta,c_r)$ in the bottom row is an equivalence.

Therefore we get two equivalent towers of spectra underneath $\Phi^{C_r} cF(|\sd_r N^\cyc_\bullet cDX|,f\Sph)$ and $F(|N^\cyc_\bullet cDX|,f\Sph)$, giving an equivalence of homotopy inverse limits
\[ \xymatrix @C=2em{
\Phi^{C_r} cF(|\sd_r N^\cyc_\bullet cDX|,f\Sph) \ar[r] \ar[d] & F(|N^\cyc_\bullet cDX|,f\Sph) \ar[d]^-\sim \\
\underset{k}\holim\, \Phi^{C_r} cF(|\sk_k \sd_r N^\cyc_\bullet cDX|,f\Sph) \ar[r]^-\sim & \underset{k}\holim\, F(|\sk_k N^\cyc_\bullet cDX|,f\Sph)  } \]
To finish proving that the top horizontal map is an equivalence, it remains to show that on the left-hand side, the derived geometric fixed points $\Phi^{C_r} (c-)$ commute with the homotopy inverse limit. This will require us to look more closely at the homotopy fibers of the maps in the homotopy limit system.

Although $\sd_r N^\cyc_\bullet cDX$ and $N^\cyc_\bullet cDX^{\sma r}$ are not isomorphic as simplicial objects, they have the same degeneracy maps and therefore have isomorphic latching maps. The cofiber of this latching map
\[ (\Sph \ra (cDX)^{\sma r})^{\square k} \square (* \ra (cDX)^{\sma r}) \]
is the smash product of $k$ copies of the $C_r$-equivariant cofiber of $\Sph \ra (cDX)^{\sma r}$ and one copy of $(cDX)^{\sma r}$. The $C_r$-equivariant dual of this is a smash product of $k$ copies of $\Sigma^\infty X^r$ and one copy of $\Sigma^\infty_+ X^r$.

To evaluate the homotopy fiber of the map of our homotopy limit system
\[ F(|\sk_k \sd_r N^\cyc_\bullet cDX|,f\Sph) \ra F(|\sk_{k-1} \sd_r N^\cyc_\bullet cDX|,f\Sph), \]
we observe that it is the dual of the cofiber of the inclusion of skeleta. By the usual latching square, this cofiber is the $k$-fold suspension of the cofiber of the latching map. Therefore our desired homotopy fiber is equivalent as a $C_r$-spectrum to
\[ \Omega^k \Sigma^\infty (X^{r})^{\sma k} \sma X^{r}_+. \]

Since $X$ is 1-connected, we can arrange so that its lowest non-basepoint cell is in dimension 2. This leads to a $C_r$-equivariant cell structure on $X^r$ in which the lowest non-basepoint cell is in the diagonal, and is also dimension 2, so $(X^{r})^{\sma k} \sma X^r_+$ has lowest non-basepoint cell in dimension $2k$. By induction on these cells, the genuine fixed points
\[ (f\Omega^k \Sigma^\infty (X^{r})^{\sma k} \sma X^{r}_+)^H \]
are at least $(k-1)$-connected, for each subgroup $H \leq C_r$. Since genuine fixed points commute with homotopy limits, we conclude that the fiber of the map from the homotopy limit to the $k$th term in the homotopy limit system
\[ F(|\sd_r N^\cyc_\bullet cDX|,f\Sph) \ra F(|\sk_k \sd_r N^\cyc_\bullet cDX|,f\Sph) \]
has $k$-connected genuine fixed points for all $H \leq C_r$.

The derived geometric fixed points of this fiber are also $k$-connected. To see this, we use an equivalent definition for the derived geometric fixed points of $E$, as the genuine fixed points of $\ti EP \sma E$ for a certain complex $\ti EP$ \cite[B.10.1]{hhr}.
Our claim then follows by induction on the cells of $\ti EP$, using the identifications
\[ (f(\Sigma^n G/H_+ \sma E))^G \simeq \Sigma^n F(G/H_+,fE)^G \simeq \Sigma^n (fE)^H. \]
In fact, this proves that for any finite $G$, a $G$-spectrum $E$ with $k$-connected genuine fixed points $(fE)^H$ for all $H \leq G$ will also have $k$-connected geometric fixed points $\Phi^H cE$ for all $H \leq G$.

Finally, since derived geometric fixed points commute with fiber sequences, we conclude that the map of derived geometric fixed points
\[ \Phi^{C_r} cF(|\sd_r N^\cyc_\bullet cDX|,f\Sph) \ra \Phi^{C_r} cF(|\sk_k \sd_r N^\cyc_\bullet cDX|,f\Sph) \]
is $(k+1)$-connected. Therefore the map to the homotopy limit is an equivalence:
\[ \Phi^{C_r} cF(|\sd_r N^\cyc_\bullet cDX|,f\Sph) \simar \underset{k}\holim\, \Phi^{C_r} cF(|\sk_k \sd_r N^\cyc_\bullet cDX|,f\Sph) \]
This finishes the proof that $\Tot(Y^\bullet) = F(|N^\cyc_\bullet cDX|,f\Sph)$ is cyclotomic.

In conclusion, our map $\Sigma^\infty_+ LX \to \Tot(Y^\bullet)$ is a map of cyclotomic spectra.
We already know that it is a stable equivalence if we ignore the circle action.
But any such equivalence of cyclotomic spectra is automatically an $\mc F$-equivalence of $S^1$ spectra, so we are done.
\end{proof}

\begin{rmk}\label{thh_d_odd_sphere_homology}
One may similarly check that this duality preserves multiplications and Adams operations. As a result, when $n \geq 1$, the homology of $THH(DS^{2n+1})$ is a tensor of a divided power algebra and an exterior algebra
\[ 
H_*(THH(DS^{2n+1})) \cong H^{-*}(LS^{2n+1}) \cong \Gamma[\alpha] \otimes \Lambda[\beta] \]
where $|\alpha| = -2n$ and $|\beta| = -(2n+1)$. 
The Adams operations $\psi^n$ are given by
\[ \psi^n(\alpha_i\beta^j) = n^i\alpha_i\beta^j. \]
\end{rmk}

\bibliographystyle{amsalpha}
\bibliography{thhdx}{}

\newcommand{\etalchar}[1]{$^{#1}$}
\providecommand{\bysame}{\leavevmode\hbox to3em{\hrulefill}\thinspace}
\providecommand{\MR}{\relax\ifhmode\unskip\space\fi MR }
\providecommand{\MRhref}[2]{%
  \href{http://www.ams.org/mathscinet-getitem?mr=#1}{#2}
}
\providecommand{\href}[2]{#2}
\begin{thebibliography}{ABG{\etalchar{+}}14}

\bibitem[ABG{\etalchar{+}}12]{angeltveit2012interpreting}
V.~Angeltveit, A.~J. Blumberg, T.~Gerhardt, M.~Hill, and T.~Lawson,
  \emph{Interpreting the {B\"o}kstedt smash product as the norm}, arXiv
  preprint arXiv:1206.4218 (2012).

\bibitem[ABG{\etalchar{+}}14]{angeltveit2014relative}
V.~Angeltveit, A.~J. Blumberg, T.~Gerhardt, M.~Hill, T.~Lawson, and M.~Mandell,
  \emph{Relative cyclotomic spectra and topological cyclic homology via the
  norm}, arXiv preprint arXiv:1401.5001 (2014).

\bibitem[AF14]{ayalapoincare}
D.~Ayala and J.~Francis, \emph{Poincar{\'e}/{K}oszul duality}, arXiv preprint
  arXiv:1409.2478 (2014).

\bibitem[Aro99]{arone_snaith}
G.~Arone, \emph{A generalization of {S}naith-type filtration}, Transactions of
  the American Mathematical Society \textbf{351} (1999), no.~3, 1123--1150.

\bibitem[BCS08]{blumberg2008topological}
A.~J. Blumberg, R.~L. Cohen, and C.~Schlichtkrull, \emph{Topological
  {H}ochschild homology of {T}hom spectra and the free loop space}, arXiv
  preprint arXiv:0811.0553 (2008).

\bibitem[BDS16]{brun2016equivariant}
M.~Brun, B.~Dundas, and M.~Stolz, \emph{Equivariant {S}tructure on {S}mash
  {P}owers}, arXiv preprint arXiv:1604.05939 (2016).

\bibitem[BG16]{barwick2016cyclonic}
C.~Barwick and S.~Glasman, \emph{Cyclonic spectra, cyclotomic spectra, and a
  conjecture of {K}aledin}, arXiv preprint arXiv:1602.02163 (2016).

\bibitem[BHM93]{bhm}
M.~B{\"o}kstedt, W.~C. Hsiang, and I.~Madsen, \emph{The cyclotomic trace and
  algebraic {K}-theory of spaces}, Inventiones mathematicae \textbf{111}
  (1993), no.~1, 465--539 (English).

\bibitem[BM11]{berger2011extension}
C.~Berger and I.~Moerdijk, \emph{On an extension of the notion of {R}eedy
  category}, Mathematische Zeitschrift \textbf{269} (2011), no.~3-4, 977--1004.

\bibitem[BM13]{blumberg2013homotopy}
A.~J. Blumberg and M.~A. Mandell, \emph{The homotopy theory of cyclotomic
  spectra}, arXiv preprint arXiv:1303.1694 (2013).

\bibitem[Boh14]{bohmann2014comparison}
A.~M. Bohmann, \emph{A comparison of norm maps}, Proceedings of the American
  Mathematical Society (2014).

\bibitem[B{\"o}k85]{bokstedt1985topological}
M.~B{\"o}kstedt, \emph{Topological {H}ochschild homology}, preprint, Bielefeld
  (1985).

\bibitem[Cam14]{campbell2014derived}
J.~A. Campbell, \emph{Derived {K}oszul {D}uality and {T}opological {H}ochschild
  {H}omology}, arXiv preprint arXiv:1401.5147 (2014).

\bibitem[CJ02]{cohen2002homotopy}
R.~L. Cohen and J.~D.~S. Jones, \emph{A homotopy theoretic realization of
  string topology}, Mathematische Annalen \textbf{324} (2002), no.~4, 773--798.

\bibitem[Con83]{connes1983cohomologie}
A.~Connes, \emph{Cohomologie cyclique et foncteurs {E}xt$^n$}, CR Acad. Sci.
  Paris S{\'e}r. I Math \textbf{296} (1983), no.~23, 953--958.

\bibitem[DHK85]{dhk}
W.~G. Dwyer, M.~J. Hopkins, and D.~M. Kan, \emph{The homotopy theory of cyclic
  sets}, Transactions of the American Mathematical Society \textbf{291} (1985),
  no.~1, 281--289.

\bibitem[HHR09]{hhr}
M.~A. Hill, M.~J. Hopkins, and D.~C. Ravenel, \emph{On the non-existence of
  elements of {K}ervaire invariant one}, arXiv preprint arXiv:0908.3724 (2009).

\bibitem[Jon87]{jones1987cyclic}
J.~D.~S. Jones, \emph{Cyclic homology and equivariant homology}, Inventiones
  mathematicae \textbf{87} (1987), no.~2, 403--423.

\bibitem[Kal10]{kaledin2010motivic}
D.~B. Kaledin, \emph{Motivic {S}tructures in {N}on-commutative {G}eometry},
  Proceedings of the International Congress of Mathematicians, vol. 901, 2010.

\bibitem[Kuh04]{kuhn2004mccord}
N.~J. Kuhn, \emph{The {McCord} model for the tensor product of a space and a
  commutative ring spectrum}, Categorical Decomposition Techniques in Algebraic
  Topology, Springer, 2004, pp.~213--235.

\bibitem[Lew78]{lewis1978stable}
L.~G. Lewis, \emph{The stable category and generalized {T}hom spectra}, Ph.D.
  thesis, University of Chicago, Department of Mathematics, 1978.

\bibitem[LMSM86]{lewis1986equivariant}
L.~G. Lewis, J.~P. May, M.~Steinberger, and J.~E. McClure, \emph{Equivariant
  stable homotopy theory}, vol. 1213, Springer-Verlag Berlin-New York, 1986.

\bibitem[Mad95]{madsen_survey}
I.~Madsen, \emph{Algebraic ${K}$-theory and traces}, Current Developments in
  Mathematics, International Press, 1995.

\bibitem[Mal14]{malkiewich2014duality}
C.~Malkiewich, \emph{Duality and linear approximations in {H}ochschild
  homology, {$K$}-theory, and string topology}, Ph.D. thesis, Stanford
  University, 2014.

\bibitem[Mal15]{malkiewich2015tower}
\bysame, \emph{A tower connecting gauge groups to string topology}, Journal of
  Topology \textbf{8} (2015), no.~2, 529--570.

\bibitem[MM02]{mandell2002equivariant}
M.~A. Mandell and J.~P. May, \emph{Equivariant orthogonal spectra and
  {S}-modules}, Memoirs of the American Mathematical Society, no. 755, American
  Mathematical Society, 2002.

\bibitem[MMSS01]{mmss}
M.~A. Mandell, J.~P. May, S.~Schwede, and B.~Shipley, \emph{Model categories of
  diagram spectra}, Proceedings of the London Mathematical Society \textbf{82}
  (2001), no.~02, 441--512.

\bibitem[Sto11]{stolz2011equivariant}
M.~Stolz, \emph{Equivariant structure on smash powers of commutative ring
  spectra}, Ph.D. thesis, University of Bergen, 2011.

\bibitem[SV02]{schwanzl2002strong}
R.~Schw{\"a}nzl and R.~M. Vogt, \emph{Strong cofibrations and fibrations in
  enriched categories}, Archiv der Mathematik \textbf{79} (2002), no.~6,
  449--462.

\end{thebibliography}

Department of Mathematics \\
University of Illinois at Urbana-Champaign \\
1409 W Green St \\
Urbana, IL 61801 \\
\texttt{cmalkiew@illinois.edu}

\end{document}